\documentclass{siamltex1213}

\usepackage{amsmath, amssymb,mathrsfs,empheq}
\usepackage{amsfonts}
\newtheorem{Case}{Case}
\newtheorem{remark}{Remark}
\newcommand{\dd}{\mathrm{d}}
\newcommand{\ep}{\varepsilon}
\newcommand{\alp}{\alpha}
\newcommand{\ff}{\hat{f}}
\newcommand{\fr}{\hat{\rho}}
\newcommand{\ii}{\mathrm{i}}
\newcommand{\e}{\mathrm{e}}

\newcommand{\dt}{\Delta t}

\title{Numerical schemes for kinetic equations in the diffusion and anomalous diffusion limits. Part I: 
the case of heavy-tailed equilibrium}
\author{ Nicolas Crouseilles
\footnotemark[2] \footnotemark[5]
\and
H\'el\`ene Hivert
\footnotemark[3]  \footnotemark[5]
\and
Mohammed Lemou
\footnotemark[4]  \footnotemark[5]
  }
\date{}

\begin{document}

\maketitle

\renewcommand{\thefootnote}{\fnsymbol{footnote}}
\footnotetext[5]{IRMAR. Universit\'e de Rennes $1$, Campus de Beaulieu. $35000$ Rennes }
\footnotetext[2]{INRIA-IPSO. Email : nicolas.crouseilles@inria.fr}
\footnotetext[3]{Email : helene.hivert@univ-rennes1.fr }
\footnotetext[4]{CNRS.  Email : mohammed.lemou@univ-rennes1.fr }

\renewcommand{\thefootnote}{\arabic{footnote}}

\begin{abstract}
In this work, we propose some numerical schemes for linear kinetic equations in the diffusion 
and anomalous diffusion limit. When the equilibrium distribution function is a Maxwellian distribution, it is well 
known that for an appropriate time scale, the small mean free path limit gives rise to a diffusion type equation. 
However, when a heavy-tailed distribution is considered, another time scale is required and 
the small mean free path limit leads to a fractional anomalous diffusion equation. Our aim is to develop 
numerical schemes for the original kinetic model which works for the different regimes, without 
being restricted by stability conditions of standard explicit time integrators. 
First, we propose some numerical schemes for the diffusion asymptotics; 
then, their extension to the anomalous diffusion limit is studied. 
In this case, it is crucial to capture the effect of the large velocities of the heavy-tailed equilibrium, 
so that some important transformations of the schemes derived for the diffusion asymptotics are needed. 
As a result, we obtain numerical schemes which enjoy the Asymptotic Preserving property in the anomalous diffusion limit, that is: they do not suffer from the restriction on the time step and they degenerate towards 
the fractional diffusion limit when the mean free path goes to zero. 
We also numerically investigate the uniform accuracy and construct a class of numerical schemes
satisfying this property. 
Finally, the efficiency of the different numerical schemes is shown through numerical experiments. 
\end{abstract}



\section{Introduction}

The modeling and numerical simulation of particles systems is a very active field of research. 
Indeed, they provide the basis for applications in neutron transport, thermal radiation, medical 
imaging or rarefied gas dynamics. According to the physical context, particles systems can be described 
at different scales. When the mean free path of the particles ({\it i.e.} the crossed distance between two collisions) 
is large compared to typical macroscopic length, 
the system is described at a microscopic level by kinetic equations; kinetic equations consider the time evolution 
of a distribution function which gives the probability of a particle to be at a given state in the six dimensional phase 
space at a given time. Conversely, when the mean free path is small, a macroscopic description 
(such as diffusion or fluid equations) can be used. It makes evolve macroscopic quantities which depends 
only on time and on the 
three dimensional spatial variable. In some situations, this description can be sufficient 
and leads to faster numerical simulations. 

Mathematically, the passage from kinetic to macroscopic models is performed by asymptotic analysis. 
From a numerical point of view, considering a small mean free path, the kinetic equation then contains stiff terms 
which make the numerical simulations very expensive for stability reasons. In fact, a typical example is the presence of 
multiple spatial and temporal scales which intervene in different positions and at different times. 
These behaviors make the construction of efficient numerical methods a real challenge.

In this work, we are interested in the time evolution of the distribution function $f$ which depends on the time $t\geq 0$, 
the space variable $x\in \Omega \subset\mathbb{R}^d$ and the velocity $v\in  \mathbb{R}^d$, with $d=1, 2, 3$. 
Particles undergo the effect of collisions which are modelized here by a linear  operator $L$ acting on $f$ 
through 
\[
L(f)(t, x, v) = \rho(t, x) M_\beta(v) - f(t, x, v), 
\]
where 
\[
\rho(t, x)=\left\langle f(t,x,v)\right\rangle =: \int_{\mathbb{R}^d} f(t, x, v) \dd v,
\] 
and 
where the equilibrium $M_{\beta}$ is defined by 
\begin{equation}
\label{equil}
M_\beta(v) =
\left\{
\begin{array}{cccll}
\displaystyle  \frac{1}{(2\pi)^{d/2}}\exp\left( -|v|^2/2  \right) & \mbox{ if } \beta = 2+d, \\
\displaystyle  \underset{|v|\to\infty}\sim\frac{m}{|v|^\beta}     & \mbox{ if } \beta \in (d,d+2), 
\end{array}
\right.
\end{equation}
where $m$ is a normalization factor.
In the sequel, we will always denote by brackets the integration over $v$. 
In order to capture a nontrivial asymptotic model, a suitable scaling has to be considered.  
The good scaling of the kinetic equation (see \cite{MelletMischlerMouhot, Mellet, Puel2}) 
satisfied by the distribution function $f$ is given by 
\begin{equation}
\label{KinDiff}
\ep^\alpha\partial_t f+\ep v\cdot \nabla_x f=L(f), \;\; \mbox{ with } \;\; \alpha=\beta-d \in (0, 2], 
\end{equation}
where $\varepsilon>0$ is the Knudsen number (ratio between the mean free path and a typical macroscopic length). The case $\alpha=2$ refers to the classical diffusion limit (the equilibrium 
corresponds to the first equation in \eqref{equil}) whereas the case 
$\alpha\in (0, 2)$ refers to the anomalous diffusion limit (the equilibrium 
corresponds to the second equation in \eqref{equil}). 
Although our analysis could be applied to general operators $L$, we will consider for a sake of simplicity  
a BGK type operator.  
We suppose that $M_\beta$ given by \eqref{equil} is an even positive function such that
$\left\langle M_\beta(v)\right\rangle =1$ 
and $\left\langle vM_\beta(v)\right\rangle=0$.
Equation \eqref{KinDiff} has to be supplemented with an initial condition $f(0,x,v)=f_0(x,v)$ and spatial periodic 
boundary conditions are considered. 

The main goal of this work is to construct numerical schemes for \eqref{KinDiff} which are uniformly stable along 
the transition from kinetic to macroscopic regime ({\it i.e.} for all $\varepsilon >0$). 
More precisely, two asymptotics are considered here: the diffusion 
limit ($\alpha=2$) and the anomalous diffusion limit ($\alpha\in (0, 2)$). For the diffusion limit, several numerical schemes  
have already been proposed but in the anomalous diffusion limit, no numerical strategy has been proposed in the literature for \eqref{KinDiff}, 
to the best of author's knowledge. 

Mathematically, the derivation of diffusion type equation from kinetic equations such as \eqref{KinDiff} when $\alpha=2$ 
was first investigated in \cite{Wigner}, \cite{BensoussanLionsPapanicolaou}, \cite{LarsenKeller} or 
\cite{DegondGoudonPoupaud}.  
When $M_\beta$ decreases quickly enough for large $|v|$, 
the solution $f$ of (\ref{KinDiff}) converges, when $\ep$ goes to zero, 
to an equilibrium function $f(t,x,v)=\rho(t,x)M_\beta(v)$ 
where $\rho(t,x)$ is solution of the diffusion equation
\begin{equation}
\label{Diff_eq}
\partial_t \rho(t,x)-\nabla_x\cdot\left( D \nabla_x \rho (t,x)\right)=0,
\end{equation}
and where the diffusion matrix $D$ is given, in the simple case of the BGK operator, by the formula
\begin{equation}
\label{D}
D=\int_{\mathbb{R}^d} v\otimes v M_\beta(v) \dd v.
\end{equation}
In particular, a crucial assumption on $M_\beta$ for $D$ to be finite is that 
\[
\int_{\mathbb{R}^d} |v|^2 M_\beta(v) \dd v <+\infty;
\] 
this is the case when $M_\beta$ is Maxwellian (first case of \eqref{equil}). 

Hence, if  the equilibrium has no finite second order moment, the asymptotic of  
rescaled kinetic equation with $\alpha=2$ is not able to capture a nontrivial dynamics 
and the value of $\alpha$ should be tuned with $M_\beta$ in order to recover this macroscopic equation 
when $\varepsilon\rightarrow 0$. Typically, let $M_\beta$ be a heavy-tailed distribution function, corresponding 
to the second case of \eqref{equil}. 
An example of such a distribution function we will often consider is $M_\beta(v)=\frac{m}{1+|v|^\beta}$ 
with $m$ chosen such that $\left\langle M_\beta(v)\right\rangle =1$.

In the case of astrophysical plasmas, it often occurs that the equilibrium $M_\beta$ is a heavy-tailed distribution of particles which typically generates an anomalous diffusion behaviour (see \cite{MendisRosenberg, SummersThorne}).
In particular, the diffusion matrix $D$ is no longer well-defined and this requires to adapt the value of $\alpha$ 
in terms of the tail decreasing rate $\beta$, 
in order to capture a nontrivial dynamics. 
It has been shown in \cite{MelletMischlerMouhot, Mellet} and \cite{Puel2} that the appropriate scaling is 
$\alpha\in (0,2)$ depends on the size $\beta$ of the tail of the equilibrium through the relation $\beta=\alp+d$.

 When $\ep$ goes to $0$, the solution of this equation  with BGK operator converges to $\rho(t,x)M_\beta(v)$ 
 where $\rho$ is the solution of the anomalous diffusion equation which can be written in Fourier variable 
 \begin{equation}
 \label{DiffAN_Eq_Four}
 \partial_t \fr(t,k)=-\kappa|k|^\alp\fr(t,k),
 \end{equation}
 where $\fr$ stands for the space Fourier variable of $\rho$, $k$ is the Fourier variable and $\kappa$ is a constant which depends only on the size of the tail of the equilibrium $M_\beta(v)$ and is defined by
 \begin{equation}
 \label{kappa}
 \kappa=\int_{\mathbb{R}^d}\frac{\left(w\cdot e\right)^2}{1+\left(w\cdot e\right)^2}\frac{m}{\left| w\right|^\beta}\dd w,
 \end{equation}
  for any $e\in\mathbb{R}^d$ such that $|e|=1$ (note that $\kappa$ does not depend on $e$). The anomalous diffusion equation can also be written in the space variable 
 \begin{equation}
 \label{DiffAN_Eq}
 \partial_t \rho(t,x)=-\frac{m\Gamma(\alp+1)}{c_{d,\alp}} \left(-\Delta_x\right)^\frac{\alp}{2}\rho(t,x),
 \end{equation}
 where $m$ is the normalization factor of $M_\beta$ 
 in (second case in \ref{equil}), 
 $\Gamma$ is the usual function defined by 
 \[
 \Gamma(x)=\int_0^{+\infty}t^{x-1}\e^{-t}\dd t,
 \]
and $c_{d,\alp}$ is a normalization constant given by 
\begin{equation}
\label{cdalp}
c_{d,\alp}=\frac{\alp \Gamma\left(\frac{d+\alp}{2}\right)}{2\pi^{\frac{d}{2}+\alp}\Gamma\left( 1-\frac{\alp}{2} \right)}.
\end{equation}
 The fractional operator of the anomalous diffusion is easily defined by its Fourier transform 
 \[
\widehat{\left(\left(-\Delta_x\right)^\frac{\alp}{2}\rho\right)}(k)=\left|k\right|^\alp \fr(k),
 \]
but has also an integral definition 
 \[
 \left(-\Delta_x\right)^{\frac{\alp}{2}}
 \rho
 (x)=c_{d,\alp}P.V.\int_{\mathbb{R}^d}\frac{\rho(x+y)-\rho(x)}{|y|^{d+\alp}}\dd y,
 \]
 where $P.V.$ denotes the principal value of the integral.
 
 The anomalous diffusion occurs not only in astrophysical plasmas but also in the study of granular media  (see \cite{ErnstBrito,BobylevGamba,BobylevCarrilloGamba}), tokamaks (see \cite{DelcastillonegreteCarrerasLynch}) or even in economy or social science: the Pareto distribution is a power tail distribution 
 that satisfies the second case of  \eqref{equil}. 
 This kind of distribution is a common object used to modelise the repartition of sizes in a given set (asteroïds, cities, income, opinions\dots), see  \cite{BouquinEco} for detailed explanations.
At a microscopic level, 
these velocity distributions arise when the motion of the particles is no longer  governed by a brownian process  but  by a Levy process (see \cite{DeMoor2}, \cite{DeMoor}). 
 
In this work, we want to construct a numerical scheme over the kinetic equation 
using a heavy-tailed distribution equilibrium, which is robust when the stiffness parameter $\ep$ goes to $0$. 
 Asymptotic Preserving (AP) schemes already exist in the case of the diffusion limit where 
 the equilibrium is a Maxwellian (see \cite{Jin,BuetCordier, CarilloGoudonLafitteVecil, JinPareschi1, Klar1,LemouMieussens,JinPareschiToscani,NaldiPareschi}) but the strategy used to obtain 
them cannot be rendered word for word to the anomalous diffusion case. 
As in the diffusion case, the difficulties when numerically solving the kinetic equation in the case 
of anomalous diffusion come from the stiff terms which require a severe condition on the numerical parameters. 
But in the anomalous diffusion case, there is an additional difficulty since one needs to take into account the large velocities arising in the equilibrium $M_\beta$. 
Indeed, the role of these high velocities is crucial and they have to be captured numerically 
in order to produce the anomalous diffusion operator when $\varepsilon\rightarrow 0$.

To derive a numerical scheme which enjoys the AP property, different strategies are investigated in this work. 
The first one is based on a fully implicit scheme in time (both the transport and the collision term are considered implicit). 
Using this {\it brute force} strategy, the so-obtained numerical scheme is obviously AP in the diffusion limit 
but not in the anomalous diffusion one. Therefore, a suitable transformation of this implicit scheme is proposed 
in this work and leads to a scheme enjoying the AP property in the anomalous diffusion limit. 
 
 From a computational point of view, a fully implicit scheme may be very expensive especially 
 if one deals with high-dimensional problems. To avoid this implicitation, we propose another strategy 
 which is based on a micro-macro decomposition. 
  The distribution function $f$ is written as a sum of 
an equilibrium part $\rho M_\beta$ plus a remainder $g$. A model equivalent to \eqref{KinDiff} is then derived, composed of a kinetic equation 
satisfied by $g$ and a macroscopic equation satisfied by $\rho$. The macroscopic flux is then reformulated 
to obtain a numerical scheme which enjoys the AP property and which does not require the implicitation of the 
transport term. 

The last strategy proposed in this work is based on a Duhamel formulation of the equation from which   
a AP numerical scheme is obtained; this approach bears similarities with \cite{HochbruckOstermann}.
This numerical scheme is efficient for the diffusion case and has to be adapted also to get an AP 
scheme in anomalous diffusion limit. 
Moreover, there are two main advantages in using such strategy. First we derive a closed equation on $\rho$ and the computation of the full distribution function is not needed at any time. 
Second, this strategy ensures that the numerical scheme has uniform accuracy in $\ep$. 
Eventually, it would be easier to construct high order (in time) schemes on this formulation, this will be done in \cite{CrousHivertLemou2}.

The paper is organized as follows. 
 In the next section we shall start with formal computations recalling  
 how the diffusion and of anomalous diffusion limits can be obtained from the kinetic equation. 
Section \ref{section3} is devoted to the derivation of the different numerical schemes: a fully implicit scheme, 
a micro-macro scheme and a scheme based on a Duhamed formulation of (\ref{KinDiff}). 
After the theoretical study of their properties, several numerical tests are presented in Section \ref{section4} 
in a one-dimensional case with periodic boundary conditions in space to illustrate and compare the efficiency  
of the various numerical schemes.

\section{A closed equation for the density and its diffusion asymptotics}
This section is devoted to the derivation from \eqref{KinDiff} of a closed equation satisfied by $\rho$.  
This equation will be the basic equation from which we formally derive the diffusion 
and the anomalous diffusion equations. 
Of course, the formal derivation of diffusion and anomalous diffusion equations is classical 
but, for a sake of clarity, we present one way (among many others) to perform such derivation. 
In particular, the formal computation below will drive the construction of a class of our numerical schemes. 
We recall here the hypothesis on the equilibrium $M_\beta(v)$ introduced in \eqref{equil}: 
it is an even function such that $\left\langle M_\beta(v)\right\rangle = 1, \left\langle vM_\beta(v)\right\rangle = 0$. 
Then we will consider the two following cases mentioned in \eqref{equil}: $\alpha=2$ and $0<\alpha<2$ 

\begin{Case}
\label{HypDiff}
$\left\langle v\otimes vM_\beta(v)\right\rangle <+\infty$ and $\alp=2$.
\end{Case}
\begin{Case}
\label{HypDiffAN}
$\left\langle  v\otimes vM_\beta(v)\right\rangle =\infty$ and $\alp\in(0,2)$. 
For example, $M_\beta$ has a heavy tail of the form 
\begin{equation}
\label{Mbeta}
M_\beta(v)\underset{|v|\to\infty}{\sim}\frac{m}{|v|^\beta}, \;\; \beta\in(d,d+2), 
\end{equation}
with a positive normalization factor $m$ ensuring $\langle M_\beta\rangle=1$.
\end{Case}

\subsection{Space Fourier transform based computations}

Starting from (\ref{KinDiff}), we propose here a method based on spatial Fourier transform 
to formally derive the asymptotic equation in both diffusion and anomalous diffusion limits. 
It will also be the basis of the numerical methods we set in Section \ref{subsec:duhamelschemes}. 

Eventually, we remind that we denote by $\hat{f}(t,k,v)$ (resp. $\fr(t,k)$) 
\[
\hat{f}(t,k,v)=\int_{\mathbb{R}^d} \e^{-\ii k\cdot v} f(t,x,v)\dd x,
\]
the space Fourier transform of $f(t,x,v)$ (resp. $\rho(t,x)$).
We have 

\begin{proposition}
\label{PropTheoriqueDiff}
\begin{romannum}
\item Equation (\ref{KinDiff}) formally implies the following exact equation on $\fr$
\begin{align}
\label{KinLim}
\fr(t,k)=&\left\langle \e^{-\frac{t}{\ep^\alp}(1+\ii\ep k\cdot v)}\ff(0,k,v)  \right\rangle
+\int_0^\frac{t}{\ep^\alp} \e^{-s}\left\langle\e^{-\ii\ep s k\cdot v}  M_\beta(v)\right\rangle\fr(t-\ep^\alp s,k)\dd s.
\end{align}
\item 
In case \ref{HypDiff} (diffusion scaling), when $\ep\to 0$, $\fr$ solves the diffusion equation
\begin{equation*}
\partial_t \fr(t,k)=-\left\langle \left( k\cdot v\right)^2M_\beta(v)\right\rangle \fr(t,k).
\end{equation*}
\item 
In case \ref{HypDiffAN} (anomalous diffusion scaling),
 when $\ep\to 0$, $\fr$ solves the anomalous diffusion equation 
\begin{equation*}
\partial_t \fr(t,k) =-\kappa\left|k\right|^\alp \fr(t,k),
\end{equation*}
with $\kappa$ defined by (\ref{kappa}) and $\beta-d=\alp.$
\end{romannum}
\end{proposition}

\newpage

\begin{proof} $ $ 
\begin{romannum}
\item \textbf{Formal derivation of the expression (\ref{KinLim}) from (\ref{KinDiff})}
\\We start by taking the Fourier transform of (\ref{KinDiff})
\[
\partial_t \ff +\ii \ep^{1-\alp} v\cdot k \ff=\frac{1}{\ep^\alp}\left(\fr  M_\beta-\ff\right). 
\]
We solve in $\ff$ (assuming $\fr$ is given) 
\[
\ff(t, k, v)=\e^{-\frac{t}{\ep^\alp}(1+\ii\ep k\cdot v)}\ff(0,k,v) +\frac{1}{\ep^\alp}\int_0^t 
\e^{-\frac{t-s}{\ep^\alp}(1+\ii\ep k\cdot v)} \fr(s,k) M_\beta(v)\dd s,
\]
which can be written after a change of variables $s\rightarrow (t-s)/\varepsilon^\alpha$
\[
\ff(t, k, v)=\e^{-\frac{t}{\ep^\alp}(1+\ii\ep k\cdot v)}\ff_0 +\int_0^\frac{t}{\ep^\alp} 
\e^{-s(1+\ii\ep k\cdot v)} \fr(t-\ep^\alp s,k) M_\beta(v)\dd s.
\]
Still denoting by brackets the integration over the velocity space, we can integrate the previous expression 
with respect to $v$ to get (\ref{KinLim}).
Then, we can  write it to make the limit equation appear 
\begin{align}
\fr=
&\hat{A}_0(t,k)
+\int_0^\frac{t}{\ep^\alp} \e^{-s}\left\langle\e^{-\ii\ep s k\cdot v}  M_\beta\right\rangle\dd s \fr
\nonumber
\\&+\ep^\alp\int_0^\frac{t}{\ep^\alp} s\e^{-s}\frac{\fr(t-\ep^\alp s,k)-\fr(t,k)}{\ep^\alp s}
\left\langle\e^{-\ii\ep s k\cdot v}  M_\beta(v)\right\rangle\dd s 
, \label{KinLim2}
\end{align}
where
\begin{equation}
\label{A0Four}
\hat{A}_0(t, k)=\left\langle \e^{-\frac{t}{\ep^\alp}(1+\ii\ep k\cdot v)}\ff_0  \right\rangle, 
\end{equation}
denotes the initial layer term.
As we will see in the next parts of the proof, this formulation gives us the two asymptotic equations of the proposition. 

\item\textbf{The diffusion limit}

Considering the diffusion scaling summarized in case \ref{HypDiff}, 
we will show that the limit equation is the diffusion equation.
In the second integral of (\ref{KinLim2}), we expand the expression of $\fr$ in a power series of $\ep$ up to the second order. 
Note that the first term of the expression is exponentially small. Indeed for all $t>0$,
$
\hat{A}_0=o(\ep^\infty) 
$
when $\ep$ goes to zero.
Then we can write, using a Taylor expansion at $\ep=0$ for the exponential 
\begin{align*}
\fr=&\int_0^\frac{t}{\ep^2} \e^{-s}\left\langle M_\beta\right\rangle\dd s\fr -\ii\ep \int_0^\frac{t}{\ep^2} s\e^{-s}\left\langle k\cdot v M_\beta\right\rangle\dd s\fr
-\frac{\ep^2}{2}\int_0^\frac{t}{\ep^2} s^2\e^{-s}\left\langle (k\cdot v)^2  M_\beta\right\rangle\dd s\fr
\\&+\ep^2\int_0^\frac{t}{\ep^2} s\e^{-s}\frac{\fr(t-\ep^2 s,k)-\fr(t,k)}{\ep^2 s}
\left\langle M_\beta\right\rangle\dd s +o(\ep^2). 
\end{align*}
Reminding that $\langle M_\beta\rangle=1$ and that $M_\beta$ is even, we can simplify this expression to obtain 
\[
\fr=\fr-\ep^2\left\langle (k\cdot v)^2M_\beta\right\rangle \fr-\ep^2\partial_t \fr+o(\ep^2),
\]
where we assume that $\fr$ is sufficiently smooth in time to write 
\[
\frac{\fr(t-\ep^2 s,k)-\fr(t,k)}{\ep^2 s}\underset{\ep\to 0}{=}-\partial_t \fr(t,k)+o(1).
\]
Finally, when $\ep$ goes to $0$ we get the expected diffusion equation 
\[
\partial_t \fr=-\left\langle (k\cdot v)^2 M_\beta\right\rangle \fr.
\]

\item \textbf{The anomalous diffusion limit}
\\Now we will show that when $M_\beta$ fulfills the conditions of case \ref{HypDiffAN} (anomalous diffusion scaling), 
the solution of (\ref{KinDiff}) converges to the solution of the anomalous diffusion equation when $\ep$ goes to zero.  

As in the case of the diffusion limit, the first term has exponential decay when $\ep$ goes to zero,
$\hat{A}_0
=o(\ep^\infty), $ for all $t>0$.
With the same assumption of smoothness in time for $\fr$, we get the following behaviour for the third term of (\ref{KinLim2}) 
\[
\int_0^\frac{t}{\ep^\alp} s\e^{-s}\frac{\fr(t-\ep^\alp s,k)-\fr(t,k)}{\ep^\alp s}
\left\langle\e^{-\ii\ep s k\cdot v}  M_\beta\right\rangle\dd s \underset{\ep\to 0}{\sim}-\partial_t \fr(t,k)+o(1). 
\]
We now have to make appear the fractional diffusion in the second term of (\ref{KinLim2}). 
Since the moment of order $2$ of $M_\beta$ is not finite, we cannot expand the exponential term 
into a power series in $\ep$ anymore. We decompose this term as
\[
\int_0^\frac{t}{\ep^\alp} \e^{-s}\left\langle\e^{-\ii\ep s k\cdot v}  M_\beta\right\rangle\dd s =
\int_0^{\frac{t}{\ep^\alp}} \e^{-s} \dd s + \int_0^\frac{t}{\ep^\alp} \e^{-s}\left\langle\left(\e^{-\ii\ep s k\cdot v} -1\right) M_\beta\right\rangle\dd s,
\]
and we focus on the second term on the right hand side of this last equality.

In order to simplify the computations, we will consider the following specific form for the equilibrium $M_\beta$: 
 $M_\beta(v)=\frac{m}{1+|v|^\beta}$, where $m$ is chosen to ensure that $\left\langle M_\beta(v)\right\rangle =1$. 
 The general case, dealing with $M_\beta$ satisfying (\ref{Mbeta}) can be done similarly, see supplementary materials for details.
  We consider the integral 
 \newline
$\left\langle \left( \e^{-\ii \ep s k\cdot v}-1\right) M_\beta\right\rangle$
and, for nonzero $k$, we perform the change of variables $w=\ep |k| v$ 
\begin{align*}
\left\langle \left( \e^{-\ii \ep s k\cdot v}-1\right) \frac{m}{1+|v|^\beta}\right\rangle
&=\left( \ep |k|\right)^{\beta-d}\int_{\mathbb{R}^d} \left( \e^{-\ii s\frac{k}{|k|} \cdot w}-1\right) \frac{m}{\left(\ep |k|\right)^\beta+|w|^\beta}\dd w. 
\end{align*}
Since the last integral has rotational symmetry, for any unitary vector $e$ of $\mathbb{R}^d$ we have
\begin{align}
\left\langle \left( \e^{-\ii \ep s k\cdot v}-1\right) \frac{m}{1+|v|^\beta}\right\rangle
&=\left( \ep |k|\right)^{\beta-d}\int_{\mathbb{R}^d} \left( \e^{-\ii s e\cdot w}-1\right) \frac{m}{\left(\ep |k|\right)^\beta+|w|^\beta}\dd w \label{CalculDiffAN_1}\\
&=\left(  \ep |k|\right)^{\beta-d} \mathcal{C}(s)+\left(  \ep |k| s\right)^{\beta-d}R(\ep,s,k) \nonumber,
\end{align}
where 
\begin{equation}
\label{Cs}
 \mathcal{C}(s)=\int_{\mathbb{R}^d} \left( \e^{-\ii s e\cdot w}-1\right)\frac{m}{|w|^\beta}\dd w,
 \end{equation}
  and 
  \[
  R(\ep,s,k)=-m\left( \ep s |k|\right)^{\beta}\int_{\mathbb{R}^d} \left( \e^{-\ii e\cdot w}-1\right)\frac{1}{|w|^\beta \left( \left(\ep s |k|\right)^\beta+|w|^\beta \right)}\dd w. 
  \]
In particular,  $R(\ep,s,k)$  tends to $0$ when $\ep$ go to zero and is bounded. If we set $R(\ep,s,0)=0$,  
equality \eqref{CalculDiffAN_1} is also true for $k=0$.

From (\ref{KinLim2}) and \eqref{CalculDiffAN_1}, we have to set $\alp=\beta-d$ to get the fractional diffusion equation when $\ep\to 0$ and then we get 
\[
\fr \underset{\ep\to 0}{=} \fr+\ep^\alp|k|^\alp\int_0^{+\infty}  \left( \int_{w\in \mathbb{R}^d}\left( \e^{-isw\cdot e}-1 \right) \frac{m}{|w|^\beta} \dd w\right)   \e^{-s} \dd s\fr -\ep^\alp \partial_t\fr+o(\ep^\alp), \; \forall t>0. 
\]
From the equality
\begin{align*}
\int_0^{+\infty}\left( \int_{w\in \mathbb{R}^d}\left( \e^{-isw\cdot e}-1 \right) \frac{m}{|w|^\beta} \dd w   \right)\e^{-s} \dd s &=-\int_{w\in\mathbb{R}^d} \frac{(w\cdot e)^2}{1+(w\cdot e)^2}\frac{m}{|w|^{d+\alp}}\dd w
:=-\kappa,
\end{align*}
we derive, when $\ep$ goes to $0$, the fractional diffusion equation (\ref{DiffAN_Eq_Four}) given in \cite{MelletMischlerMouhot,Puel2}.

\end{romannum}
\qquad \end{proof}

\subsection{Computations in the space variable}

The aim of this section is to generalize the previous computations to the original space variable.
We have the following proposition  

\begin{proposition}
\label{PropTheoriqueDiffAN}
\begin{romannum}
\item Equation (\ref{KinDiff}) formally implies the following exact equation on $\rho$
\begin{equation}
\label{KinLimNoFour}
\rho(t,x)=\left\langle \e^{-{\frac{t}{\ep^\alp}}}f_0(x-\ep^{1-\alp}tv,v) \right\rangle +\int_0^{\frac{t}{\ep^\alp}}\e^{-s} \left\langle \rho(t-\ep^\alp s, x-\ep s v) M_\beta(v) \right\rangle \dd s.
\end{equation}
\item In case \ref{HypDiff} (diffusion scaling), when $\ep\to 0$, $\rho$ solves the diffusion equation 
\[
\partial_t \rho(t,x)=\nabla_x\cdot\left( D\nabla_x \rho(t,x) \right),
\text{~with~} D=\left\langle v\otimes vM_\beta(v)\right\rangle.
\] 
\item
In  case \ref{HypDiffAN} (anomalous diffusion scaling), when $\ep\to 0$, $\rho$ solves the anomalous diffusion equation 
\[
\partial_t \rho(t,x)=-\frac{m\Gamma(\alp+1)}{c_{d,\alp}} \left(-\Delta_x\right)^\frac{\alp}{2}\rho(t,x),
\]
where $\Gamma$ is the usual Euler Gamma function, $m$ the normalization constant appearing in 
the expression of the equilibrium $M_\beta$, $c_{d,\alp}$ is defined by (\ref{cdalp}) and
$\beta-d=\alp$.

\end{romannum}
\end{proposition}

\begin{proof} $ $

\begin{romannum}
\item \textbf{Formal derivation of the expression (\ref{KinLimNoFour})  from (\ref{KinDiff})}

As in the previous section, we can integrate (\ref{KinDiff}) to get  
%
\[
\rho=A_0 +\int_0^{\frac{t}{\ep^\alp}}\e^{-s} \left\langle \rho(t-\ep^\alp s, x-\ep s v) M_\beta \right\rangle \dd s,
\]
with
\begin{equation}
\label{A0}
A_0(t,x)=\left\langle \e^{-{\frac{t}{\ep^\alp}}}f_0(x-\ep^{1-\alp}tv,v) \right\rangle.
\end{equation}

\item \textbf{The diffusion limit}
\\Since $\langle\rho(t,x-\ep s v)\rangle =\langle \rho(t,x+\ep s v)\rangle$, we can rewrite (\ref{KinLimNoFour}) 
in order to make the limit equation appear 
\begin{align}
\label{KinDiffNoFour}
\rho(t,x)&=A_0(t,x)+\ep^2\int_0^\frac{t}{\ep^2}s\e^{-s}\left\langle \frac{\rho(t-\ep^2s,x-\ep sv)-\rho(t,x-\ep s v)}{\ep^2 s} M_\beta(v)\right\rangle \dd s \nonumber \\&
+\frac{\ep^2}{2}\int_0^{\frac{t}{\ep^2}}s^2\e^{-s}\left\langle \frac{\rho(t,x-\ep v s)+\rho(t,x+\ep v s)-2\rho(t,x)}{\ep^2  s^2} M_\beta(v) \right\rangle \dd s
 \nonumber \\ 
 &+\int_0^{\frac{t}{\ep^2}} \e^{-s}\dd s \rho(t,x).
\end{align}
First, we again notice that the initial layer term $A_0$ defined in (\ref{A0}) has exponential decay 
when $\ep$ goes to zero. Second, Taylor expansions in the two integrals in time enable to write for $\ep$ sufficiently small
\begin{align*}
\rho&=\ep^2\int_0^{\frac{t}{\ep^\alp}} s \e^{-s}\left( -\partial_t \rho \right)\left\langle M_\beta\right\rangle\dd s+\frac{\ep^2}{2} \int_0^{\frac{t}{\ep^\alp}} s^2 \e^{-s} \left\langle {}^t v H_x\rho v M_\beta \right\rangle \dd s 
\\&+\left( 1-\e^{-\frac{t}{\ep^2}} \right)\rho +o(\ep^2),
\end{align*}
where $H_x\rho$ denotes the Hessian of $\rho$. This last expression can be simplified into
\[
\partial_t \rho=\left\langle {}^t v H_x\rho v M_\beta(v) \right\rangle,
\]
when $\ep$ goes to $0$. With $D$ defined in (\ref{D}), we have 
$\left\langle {}^t v H_x\rho v M_\beta(v) \right\rangle=\nabla_x\cdot\left( D\nabla_x \rho \right)$ 
so this last equation is the diffusion equation as expected.

\item \textbf{The anomalous diffusion limit}
\\We rewrite (\ref{KinLimNoFour}) in the case of the anomalous diffusion scaling as follows
\begin{align}
\rho(t,x)&=A_0(t,x)+\ep^\alp\int_0^{\frac{t}{\ep^\alp}}s\e^{-s}\left\langle  \frac{\rho\left(t-\ep^\alp s,x-\ep v s\right)-\rho\left( t,x-\ep v s\right)}{\ep^\alp s} M_\beta(v)\right\rangle \dd s \nonumber \\
&+\int_0^{\frac{t}{\ep^\alp}}\e^{-s}\left\langle \left( \rho(t,x-\ep v s)-\rho(t,x)\right)M_\beta(v)\right\rangle \dd s  +\int_0^{\frac{t}{\ep^\alp}}\e^{-s}\dd s \left\langle M_\beta(v)\right\rangle \rho(t,x). \label{KinDiffANNoFour}
\end{align}
Under smoothness assumptions on $\rho$, the first integral of (\ref{KinDiffANNoFour}) 
becomes, for small $\ep$
\begin{align*}
\int_0^{\frac{t}{\ep^\alp}}s\e^{-s}\left\langle  \frac{\rho\left(t-\ep^\alp s,x-\ep v s\right)-\rho\left( t,x-\ep v s\right)}{\ep^\alp s} M_\beta(v)\right\rangle \dd s 
&=-\partial_t \rho(t,x) +o(1).
\end{align*}
Since the last term of (\ref{KinDiffANNoFour}) can be explicitly computed, let us focus on the second integral of (\ref{KinDiffANNoFour}). We perform the change of variables $y=x-\ep v s$ in the integral over $v$  
and exchanging the integrals in $t$ and $y$ to get 
\begin{equation}
\label{apparitiondea}
\int_0^{\frac{t}{\ep^\alp}}\e^{-s}\left\langle \left( \rho(t,x-\ep v s)-\rho(t,x)\right)M_\beta(v)\right\rangle \dd s 
=\int_{\mathbb{R}^d} \frac{\rho(t,y)-\rho(t,x)}{\left| x-y\right|^\beta}  a(\ep,x-y) \dd y,  
\end{equation}
where
\begin{equation}
\label{a}
a(\ep,z)=\int_0^{\frac{t}{\ep^\alp}} |z|^\beta \frac{\e^{-s}}{(\ep s)^d}M_\beta\left( \frac{z}{\ep s} \right) \dd s.
\end{equation}
Now, we want to find an equivalent of $a(\ep,z)$ for small values of $\ep$. Here we consider the equilibrium $M_\beta(v)=\frac{m}{1+|v|^\beta}$, the general case with $M_\beta$ satisfying (\ref{Mbeta}) can be done similarly (see supplementary materials for more details). The integral $a$ becomes 
\[
a(\ep,z)=m\int_0^{\frac{t}{\ep^\alp}} |z|^\beta \frac{(\ep s)^\beta}{(\ep s)^\beta+|z|^\beta}\frac{\e^{-s}}{(\ep s)^d}\dd s,
\]

and we consider, for a nonzero $z$, the following quantity
\begin{align*}
a(\ep,z)-m\ep^{\beta-d}\Gamma(\beta-d+1)
&=-m\int_0^{+\infty} (\ep s)^{\beta-d}\e^{-s} \frac{(\ep s)^\beta}{(\ep s)^\beta+|z|^\beta}\dd s  \\
&-m\int_{\frac{t}{\ep^\alp}}^{+\infty} |z|^\beta \frac{(\ep s)^\beta}{(\ep s)^\beta+|z|^\beta}\frac{\e^{-s}}{(\ep s)^d}\dd s.
\end{align*}
Then,  for almost every $z\in\mathbb{R}^d$, $\frac{1}{\ep^{\beta-d}}\left(a(\ep,z)-m\ep^{\beta-d}\Gamma(\beta-d+1)\right)$ converges to $0$ when $\ep$ goes to zero. We also have the following estimation 
\begin{equation}
\label{inegalitecool}
\left|a(\ep,z)-m\ep^{\beta-d}\Gamma(\beta-d+1)\right|
\le K \left( \ep^{\beta-d}+ \ep^{\beta-d}\e^{-\frac{t}{2\ep^\alp}}\right).
\end{equation}

Let us remark that  for a reasonable regularity of $\rho$, the integral appearing in (\ref{apparitiondea})
is well defined as a principal value because $\beta\in(d,d+2)$
\[
P.V.\int_{\mathbb{R}^d} \frac{\rho(t,y)-\rho(t,x)}{|y-x|^\beta}\dd y=\int_{\mathbb{R}^d} \frac{\rho(t,y)-\rho(t,x)-(y-x)\cdot \nabla_x \rho(t,x)}{|y-x|^\beta}\dd y.
\]
Hence, using (\ref{inegalitecool}) we obtain that for almost all $y\in\mathbb{R}^d$
\begin{align*}
&\left|\frac{1}{\ep^{\beta-d}} \frac{\rho(t,y)-\rho(t,x)-(y-x)\cdot \nabla_x \rho(t,x)}{|y-x|^\beta}\left( a(\ep,z)-m\ep^{\beta-d}\Gamma(\beta-d+1) \right)\right|, 
\end{align*}
 is dominated by an integrable function. We also obtained that for almost all $y\in\mathbb{R}^d$ 
 it tends to zero when $\ep$ goes to zero, by using the dominated convergence theorem for this function.

We have, from (\ref{apparitiondea})
\begin{align*}
\int_0^{\frac{t}{\ep^\alp}}&\e^{-s}\left\langle \left( \rho(t,x-\ep v s)-\rho(t,x)\right)M_\beta(v)\right\rangle \dd s \\
&=\int_{\mathbb{R}^d} \frac{\rho(t,y)-\rho(t,x)-(y-x)\cdot \nabla_x \rho(t,x)}{|x-y|^\beta} a(\ep, x-y)\dd y,
\end{align*}
where we used that $a$ is even. 
So we have 
\begin{align*}
&\int_0^{\frac{t}{\ep^\alp}}\e^{-s}\left\langle \left( \rho(t,x-\ep v s)-\rho(t,x)\right)M_\beta(v)\right\rangle \dd s \\
&\underset{\ep\to 0}\sim m\ep^{\beta-d}\Gamma(\beta-d+1)\int_{\mathbb{R}^d} \frac{\rho(t,y)-\rho(t,x)-(y-x)\nabla_x\rho(t,x)}{|y-x|^\beta}\dd y \\
&\underset{\ep\to 0}\sim m\ep^{\beta-d}\Gamma(\beta-d+1)P.V.\int_{\mathbb{R}^d} \frac{\rho(t,y)-\rho(t,x)}{|y-x|^\beta}\dd y.
\end{align*}
We remark that we have to set
$\beta-d=\alp$ to recover the limit equation, so we get
\[
\rho(t,x)\underset{\ep\to 0}=-\ep^\alp \partial_t\rho(t,x)+ m\ep^{\alp}\Gamma(\alp+1)P.V.\int_{\mathbb{R}^d} \frac{\rho(t,y)-\rho(t,x)}{|y-x|^{\alp+d}}\dd y +\rho(t,x)+o\left(\ep^{\alp}\right),
\]
that is when $\ep$ goes to $0$
\[
\partial_t\rho(t,x)=-m\Gamma(\alp+1)P.V.\int_{\mathbb{R}^d} \frac{\rho(t,x)-\rho(t,y)}{|y-x|^{\alp+d}}\dd y,
\]
which is the expected anomalous diffusion equation. 
\end{romannum}
\qquad \end{proof}

\section{Numerical schemes}
\label{section3}

This section is devoted to the presentation of appropriate numerical schemes to capture the solution of (\ref{KinDiff}).
We want these schemes to be Asymptotic Preserving (AP), that is, for arbitrary initial condition $f_0$
\begin{remunerate}
\item to be consistent with  the kinetic equation (\ref{KinDiff}) when the time step $\dt$ tends to $0$, with a fixed $\ep$,
\item to degenerate into a scheme solving the asymptotic equation (diffusion or anomalous diffusion) when $\ep$ goes to $0$ with a fixed time step $\dt$.
\end{remunerate}
In the sequel, three different numerical schemes are presented: a fully implicit scheme, a micro-macro decomposition based scheme and a Duhamel based scheme. We will consider a time discretization $t_n=n\dt, n=0,\ldots, N$ such that $N\dt= T$ where $T$ is the final time and we will set $f^n\simeq f(t_n)$.

As we used Fourier transform in the previous formal computations, we will consider a bounded domain $\Omega$ for the spatial domain with periodic conditions. Hence, we will be able to use the discrete Fourier transform in the algorithm.

\subsection{Implicit scheme}

The first idea to design a scheme for the kinetic equation (\ref{KinDiff}) is to set an implicit scheme over the Fourier formulation of the kinetic equation. We will see that in the case of the diffusion limit it preserves easily the asymptotic equation. But in the case of the anomalous diffusion limit, the effect of large velocities must be taken into account to obtain the good asymptotic and this needs a suitable modification of this fully implicit scheme.

We start with (\ref{KinDiff}) written in the spatial Fourier variable
and consider a fully implicit time discretization 
\[
\frac{\ff^{n+1}-\ff^n}{\Delta t}+\frac{1}{\ep^\alp}\left(1+\ii\ep k\cdot v\right)\ff^{n+1}=\frac{1}{\ep^\alp}\fr^{n+1}M_\beta. 
\]
We have 
\begin{equation}
\label{fnplus1}
\ff^{n+1}=\frac{1-\lambda}{1+\ii\lambda\ep k\cdot v}\ff^n+\frac{\lambda}{1+\ii\lambda\ep k\cdot v}\fr^{n+1}M_\beta,
\end{equation}
with
\begin{equation}
\label{lambda}
\lambda=\frac{\dt}{\varepsilon^\alpha+\dt}. 
\end{equation}
Note that $0<\lambda<1, \lambda\underset{\dt\to 0}\longrightarrow0$ and that $\lambda\underset{\ep\to 0}{\rightarrow}1$.

To compute $\ff^{n+1}$ from \eqref{fnplus1}, $\fr^{n+1}$ has to be determined first. 
To do that 
we integrate (\ref{fnplus1}) with respect to $v$ to get  
\begin{equation}
\label{Implicite1}
\fr^{n+1}=\lambda\left\langle \frac{M_\beta}{1+\ii \lambda\ep k\cdot v} \right\rangle \fr^{n+1}+\left(1-\lambda\right)\left\langle\frac{\ff^n}{1+\ii\lambda\ep k\cdot v}\right\rangle,
\end{equation}
that is
\begin{equation}
\label{rhonplus1}
\fr^{n+1}=\left\langle \frac{\ff^n}{1+\ii\lambda\ep k\cdot v} \right\rangle\frac{1}{\left\langle \frac{M_\beta}{1+\ii\lambda\ep k\cdot v} \right\rangle +\frac{1}{1-\lambda}\left\langle \frac{\ii \lambda\ep k\cdot v M_\beta}{1+\ii \lambda\ep k\cdot v} \right\rangle}.
\end{equation}

We have the following proposition
\begin{proposition}
\label{PropFullImplicitDiff}
In the case \ref{HypDiff} (diffusion scaling), we consider the scheme defined for all $k$ and 
for all time index $0\le n\leq N, N\dt = T$ by
\begin{equation}
\label{S1}
\left\{
\begin{array}{l} 
\displaystyle
\fr^{n+1}(k)=\left\langle \frac{\ff^n}{1+\ii\lambda\ep k\cdot v} \right\rangle
\left(\left\langle \frac{M_\beta}{1+\ii\lambda\ep k\cdot v} \right\rangle +\frac{1}{1-\lambda}\left\langle \frac{\ii \lambda\ep k\cdot v M_\beta}{1+\ii \lambda\ep k\cdot v} \right\rangle\right)^{-1} \\
\displaystyle
\ff^{n+1}=\frac{1}{1+\ii\lambda\ep k\cdot v} \left[ (1-\lambda) \ff^n+\lambda \fr^{n+1}M_\beta\right],
\end{array}
\right.
\end{equation}
with $\lambda=\dt/(\varepsilon^2+\dt)$ and with the initial condition $\ff^0(k,v)=\ff_0(k,v)$. 
This scheme has the following properties:
\begin{romannum}
\item \label{31i} The scheme is of order $1$ for any fixed $\ep$ and  preserves the total mass
\[
\forall n\in[\![1,N]\!], \fr^n(0)=\fr^0(0).
\]
\item \label{31ii} The scheme is AP: 
for a fixed $\dt$, the scheme solves the diffusion equation when $\ep$ goes to zero 
	\begin{equation}
	\label{EulerDiffFour}
\frac{\fr^{n+1}(k)-\fr^n(k)}{\Delta t}=-\left\langle (k\cdot v)^2 M_\beta(v)\right\rangle\fr^{n+1}(k).
\end{equation}
\end{romannum}

\end{proposition}

\begin{proof} As (i) is straightforward, let us prove (ii).
In the case of the diffusion asymptotic, we must check that the scheme is AP. 
We rewrite (\ref{rhonplus1}) to get
\[
\left(1-\lambda\right)\fr^{n+1}=-\lambda\left\langle \frac{\ii\lambda\ep k\cdot v M_\beta}{1+\ii\lambda\ep k\cdot v} \right\rangle \fr^{n+1}+\left(1-\lambda\right)\left\langle \frac{\ff^n}{1+\ii\lambda\ep k\cdot v}\right\rangle,
\]
that is, with the equality $\frac{1}{\dt}\left(\frac{\lambda}{1-\lambda}\right)=\frac{1}{\ep^2},$
\begin{equation}
\label{Implicite2}
\frac{\fr^{n+1}-\fr^n}{\dt}=-\frac{1}{\ep^2}\left\langle \frac{\ii\ep \lambda k\cdot v}{1+\ii\lambda\ep k\cdot v} M_\beta\right\rangle\fr^{n+1}-\frac{1}{\dt}\left\langle \frac{\ii\lambda\ep k\cdot v}{1+\ii\lambda\ep k\cdot v}\ff^n\right\rangle.
\end{equation}
 Hence 
 we can write 
 \[
 \frac{\fr^{n+1}-\fr^n}{\dt}=-\left\langle \frac{\left(k\cdot v\right)^2\lambda^2}{1+\lambda^2\ep^2(k\cdot v)^2}M_\beta \right\rangle \fr^{n+1}-\frac{1}{\dt}\left\langle \frac{\ii\lambda\ep k\cdot v}{1+\ii\lambda\ep k\cdot v}\ff^n\right\rangle,
 \]
 and so,  when $\ep\to 0$ with a fixed $\Delta t$, the scheme degenerates into
\[
\frac{\fr^{n+1}-\fr^n}{\Delta t}\underset{\ep\to 0}=-\left\langle (k\cdot v)^2M_\beta(v)\right\rangle \fr^{n+1},
\]
that is an implicit scheme for  (\ref{Diff_eq}) in Fourier variable. 

\qquad \end{proof}

Now we consider the case of an anomalous diffusion scaling $0<\alpha<2$. In this case, 
the equilibrium $M_\beta$ is heavy-tailed and 
$D=\left\langle v\otimes vM_\beta(v)\right\rangle =+\infty$. However, we have  
$$
\lim\limits_{\ep\to 0}\frac{1}{\ep^\alp}\left\langle \frac{\ep^2\left(k\cdot v\right)^2\lambda^2}{1+\lambda^2\ep^2(k\cdot v)^2}M_\beta(v) \right\rangle \fr^{n+1}=\left| k\right|^\alp \kappa, 
$$ 
with $\kappa$ defined in (\ref{kappa}). Unfortunately a discretization in velocity of the integral involved in this limit does not see the heavy tail of the equilibrium, which means that the effect of large velocities is completely missed. 
In fact, the limit when $\ep$ goes to zero of the implicit kinetic scheme we wrote in Prop. \ref{PropFullImplicitDiff} degenerates into the following scheme 
$$
\fr^{n+1} = \fr^n, \;\; \ff^{n+1} = \fr^{n+1}M_\beta, 
$$
which is not the correct asymptotics. 
Hence we have to transform its expression to make the right limit clearly appear in the scheme. 
The scheme is presented in the following proposition. 

\begin{proposition}
\label{PropFullImplicitDiffAN}
In the case \ref{HypDiffAN}  (anomalous diffusion scaling), we consider 
the following scheme defined for all $k$ and for all time index $0\le n\le N, N\dt=T$ by
\[
\left\{
\begin{array}{l}
\displaystyle
\fr^{n+1}(k)=\frac{\displaystyle\left\langle \frac{\ff^n}{1+\ii\lambda\ep k\cdot v} \right\rangle}{\displaystyle\left\langle \frac{M_\beta}{1+\ii\lambda\ep k\cdot v} \right\rangle +\frac{\left( \ep \lambda |k|\right)^\alp}{1-\lambda}\left\langle    \frac{m}{\left( \ep \lambda |k| \right)^{\alp+d}+|v|^{\alp+d}} \frac{\left( v\cdot  e \right)^2}{ 1+\left( v\cdot e \right)^2}     \right\rangle   } \\
\displaystyle
\ff^{n+1}=\frac{1}{1+\ii\lambda\ep k\cdot v} \left[ (1-\lambda) \ff^n+\lambda \fr^{n+1}M_\beta\right],
\end{array}
\right.
\]
with $\lambda=\dt/(\varepsilon^\alpha+\dt)$, where $e$ is any unitary vector and with the initial condition 
$\ff^0(k,v)=\ff_0(k,v)$. 
This scheme has the following properties:
\begin{romannum}
\item The scheme is of order $1$ for any fixed $\ep$ and preserves the total mass 
\[
\forall n\in[\![1,N]\!], \fr^n(0)=\fr^0(0).
\]
\item The scheme is AP: 
for a fixed $\dt$, the scheme solves the anomalous diffusion equation when $\ep$ goes to zero 
	\begin{equation}
	\label{EulerDiffANFour}
\frac{\fr^{n+1}(k)-\fr^{n}(k)}{\Delta t}=-\kappa |k|^\alp \fr^{n+1}(k),
\end{equation}
where $\kappa$ has been defined by (\ref{kappa}).
\end{romannum}
\end{proposition}

\begin{proof}
As (i) is straightforward, let us prove (ii).
We start by considering the expression of $\fr^{n+1}$ in the scheme of Prop. \ref{PropFullImplicitDiff}, and more precisely the integral
\[
\left\langle \frac{\ii  \ep\lambda  k\cdot v~M_\beta}{1+\ii\ep\lambda k\cdot v } \right\rangle
=\int_{\mathbb{R}^d}M_\beta(v)\frac{\ep^2\lambda^2\left( k\cdot v \right)^2}{1+\ep^2\lambda^2\left( k\cdot v \right)^2}\dd v. 
\]
As in the previous section, we perform the change of variables $w=\ep\lambda|k| v$ to obtain 
\begin{align}
\int_{\mathbb{R}^d}M_\beta(v)\frac{\ep^2\lambda^2\left( k\cdot v \right)^2}{1+\ep^2\lambda^2\left( k\cdot v \right)^2}\dd v
&=\left( \ep\lambda |k|\right)^{-d} \int_{\mathbb{R}^d}M_\beta\left( \frac{w}{\ep \lambda|k|} \right)
\frac{\left( w\cdot \frac{k}{|k|} \right)^2}{ 1+\left( w\cdot \frac{k}{|k|} \right)^2}\dd w, \label{CalculDiffAN_2} \\
&=\left( \ep \lambda|k|\right)^{-d} \int_{\mathbb{R}^d} M_\beta\left( \frac{w}{\ep \lambda|k|} \right)\frac{\left( w\cdot  e \right)^2}{ 1+\left( w\cdot e \right)^2}\dd w,  \nonumber 
\end{align}
with $e=k/|k|$; note that the last integral does not depend on this unitary vector $e$ thanks to 
its rotational invariance. 

We consider explicitly the case  $M_\beta(v)=\frac{m}{1+|v|^{\alp+d}}$ with $m$ such that 
$\langle M_\beta\rangle=1$ and with $\alp\in]0,2[$. We have
\[
\left\langle M_\beta \frac{ \ep^2\lambda^2\left( k\cdot v\right)^2}{1+\ep^2\lambda^2\left( k\cdot v \right)^2} \right\rangle
=\left( \ep \lambda |k|\right)^\alp \int_{\mathbb{R}^d} \frac{m}{\left( \ep \lambda |k| \right)^{\alp+d}+|w|^{\alp+d}} \frac{\left( w\cdot  e \right)^2}{ 1+\left( w\cdot e \right)^2}\dd w,
\]
where $e$ is any unitary vector. Inserting this formula in \eqref{S1}, 
we get the expression for $\fr^{n+1}$. The equation satisfied by $\ff^{n+1}$ in Prop. \ref{PropFullImplicitDiff} 
is unchanged, then the numerical scheme of the proposition is derived. 
Now, it remains to show that this scheme enjoys the AP property. 
From the following form of the equation on $\fr^{n+1}$
\begin{equation}
\label{Implicite3}
\frac{\fr^{n+1}-\fr^n}{\dt}=-\lambda^\alp |k|^\alp \left\langle \frac{m}{\left( \ep \lambda |k| \right)^{\alp+d}+|v|^{\alp+d}} \frac{\left( v\cdot  e \right)^2}{ 1+\left( v\cdot e \right)^2}\right\rangle\fr^{n+1}-\left\langle \frac{\ii\lambda\ep k\cdot v}{1+\ii\lambda\ep k\cdot v}\frac{\ff^n}{\dt}\right\rangle,
\end{equation}
we easily observe that it degenerates when $\ep$ goes to zero to an implicit scheme for the anomalous diffusion equation. This concludes the proof.

\qquad \end{proof}

\subsection{Scheme based on a micro-macro decomposition}
\label{SectionMMscheme}
In the previous part, we constructed a fully implicit scheme enjoying the AP property. The implicit character of the transport may induce a high computational cost. Therefore, we propose another scheme, which is 
based on a micro-macro decomposition of the kinetic equation and in which the transport part is explicit in time. 
In the diffusion case, such a scheme has been set in \cite{LemouMieussens}. 
We first recall here the way to derive it in the diffusion case. Then, we consider 
the case of the anomalous diffusion and show how to take into account the effect of the large velocities 
induced by the heavy-tailed structure of the equilibrium $M_\beta(v)$. 

The distribution $f$ is decomposed as 
\[
f=\left\langle f\right\rangle M_\beta+g=\rho M_\beta+g,
\]
where 
$\left\langle g\right\rangle=0$. Injecting this decomposition into (\ref{KinDiff}), we have 
\begin{equation}
\label{mm1}
\partial_t \rho \, M_\beta+\partial_t g+\frac{\ep}{\ep^\alp}v\cdot \nabla_x\rho \, M_\beta+\frac{\ep}{\ep^\alp}v\cdot \nabla_x g=-\frac{1}{\ep^\alp}g.
\end{equation}
To derive an equation on $\rho$, we integrate with respect to $v$ to obtain 
\begin{equation}
\label{mm1_part1}
\partial_t\rho+\frac{\ep}{\ep^\alp}\left\langle v\cdot \nabla_x g\right\rangle =0.
\end{equation}
Now we replace $\partial_t \rho$ in  (\ref{mm1}) to get an equation satisfied by $g$
\begin{equation}
\label{mm1_part2}
\partial_t g+\frac{\ep}{\ep^\alp}v\cdot\nabla_x\rho \, M_\beta+\frac{\ep}{\ep^\alp}\left( v\cdot \nabla_xg-\left\langle v\cdot \nabla_xg\right\rangle M_\beta  \right)=-\frac{1}{\ep^\alp}g.
\end{equation}
From this formulation 
the semi-implicit scheme proposed in \cite{LemouMieussens} writes 
\begin{equation}
\label{mmschemeDiff}
\left\{
\begin{matrix}
\displaystyle
\frac{\rho^{n+1}-\rho^n}{\dt}+\frac{\ep}{\ep^\alp}\left\langle v\cdot \nabla_x g^{n+1} \right\rangle=0 \\
\displaystyle
\frac{g^{n+1}-g^n}{\dt}+\frac{\ep}{\ep^\alp}v\cdot \nabla_x \rho^n M_\beta+\frac{\ep}{\ep^\alp}\left( v\cdot \nabla_x g^n-\left\langle v\cdot \nabla_x g^n \right\rangle M_\beta \right)=-\frac{1}{\ep^\alp}g^{n+1},
\end{matrix}
\right.
\end{equation}
and we have the following proposition
\begin{proposition}
\label{PropMicroMacroDiff}
In case \ref{HypDiff} (diffusion scaling), we consider the following scheme defined for all $x\in \Omega$, $v\in V$ and all time index $0\le n\le N$ by 
\[
\left\{
\begin{array}{l}
\displaystyle
\frac{\rho^{n+1}-\rho^n}{\dt}+\frac{1}{\ep}\left\langle v\cdot \nabla_x g^{n+1} \right\rangle=0 \\
\displaystyle
\frac{g^{n+1}-g^n}{\dt}+\frac{1}{\ep}v\cdot \nabla_x \rho^n M_\beta+\frac{1}{\ep}\left( v\cdot \nabla_x g^n-\left\langle v\cdot \nabla_x g^n \right\rangle M_\beta \right)=-\frac{1}{\ep^2}g^{n+1},\end{array}
\right.
\]
with initial conditions $\rho^0(x)=\rho(0,x)$ and $g^0(x,v)=f(0,x,v)-\rho(0,x)M_\beta(v)$. 
This scheme has the following properties:
\begin{romannum}
\item The scheme is of order $1$ for any fixed $\ep$ and preserves the total mass
\[
\forall n\in[\![1,N]\!], \fr^n(0)=\fr^0(0).
\]
\item The scheme is AP: 
for a fixed $\dt$, the scheme solves the diffusion equation when $\ep$ goes to zero 
$$
\frac{\fr^{n+1}(k)-\fr^n(k)}{\Delta t}=-\left\langle (k\cdot v)^2 M_\beta(v)\right\rangle\fr^{n}(k).
$$
\end{romannum}

\end{proposition}

For the proof, we refer to \cite{LemouMieussens}.

We now consider the anomalous diffusion scaling. Formulation (\ref{mmschemeDiff}) does not work because the fractional laplacian does not appear when $\ep$ goes to $0$, so we have to modify the micro-macro scheme. As $M_\beta$ is even we have
\[
\left\langle v\cdot\nabla_x g^{n+1}\right\rangle = \left\langle v\cdot\nabla_x \rho^{n+1}M_\beta\right\rangle+\left\langle v\cdot\nabla_x g^{n+1}\right\rangle =\left\langle v\cdot\nabla_x f^{n+1}\right\rangle,
\]
and so the first equation of  (\ref{mmschemeDiff}) can be rewritten as
\[
\frac{\rho^{n+1}-\rho^n}{\dt}+\frac{\ep}{\ep^\alp}\left\langle v\cdot \nabla_x f^{n+1}\right\rangle =0.
\]
 Now we use an implicit formulation of the kinetic equation to express $f^{n+1}$
\[
\ep^\alp \frac{f^{n+1}-f^n}{\dt}+\ep v\cdot\nabla_x f^{n+1}=\rho^{n+1}M_\beta-f^{n+1}.
\]
As we did in the previous part we introduce the variable $\lambda$ defined by (\ref{lambda}) 
and write this last equation as  
\[
\left( I+\ep\lambda v\cdot \nabla_x \right) f^{n+1}=\lambda \rho^{n+1}M_\beta+(1-\lambda)f^n,
\]
that is
\[
f^{n+1}=\lambda\left(I+\ep\lambda v\cdot\nabla_x\right)^{-1}\rho^{n+1}M_\beta+(1-\lambda)\left(I+\ep\lambda v\cdot \nabla_x\right)^{-1}f^{n}.
\]
As $\lambda\underset{\dt\to 0}= O(\dt)$, we will use the following approximated expression for $f^{n+1}$
\[
f^{n+1}=\lambda\left(I+\ep\lambda v\cdot\nabla_x\right)^{-1}\rho^{n+1}M_\beta+(1-\lambda)\left(I-\ep\lambda v\cdot \nabla_x\right)f^{n}+O(\dt),
\]
to avoid a costly inversion of the transport operator. In this expression, the terms we removed also 
vanish when $\ep$ goes to zero. 
The scheme  presented in Prop. \ref{PropMicroMacroDiff} can be rewritten in the anomalous diffusion limit as 
\begin{equation}
\label{mmscheme2}
\left\{
\begin{array}{l l l}
\displaystyle
\frac{\rho^{n+1}-\rho^n}{\dt}&\displaystyle+\frac{\ep}{\ep^\alp}\lambda\left\langle v\cdot \nabla_x \left(  I+\ep\lambda v\cdot \nabla_x  \right)^{-1}\rho^{n+1}M_\beta \right\rangle&  \\
&\displaystyle+\frac{\ep}{\ep^\alp}(1-\lambda)\left\langle v\cdot \nabla_x \left(  I-\ep\lambda v\cdot \nabla_x  \right)f^{n} \right\rangle=0,&  \\
\displaystyle
\frac{g^{n+1}-g^n}{\dt}&+\frac{\ep}{\ep^\alp}v\cdot \nabla_x \rho^n M_\beta+\frac{\ep}{\ep^\alp}\left( v\cdot \nabla_x g^n-\left\langle v\cdot \nabla_x g^n \right\rangle M_\beta \right)=-\frac{1}{\ep^\alp}g^{n+1}.&
\end{array}
\right.
\end{equation}
We can make additional simplifications. First, we remark that, 
$$
\left\langle v\cdot \nabla_x \left(  I-\ep\lambda v\cdot \nabla_x  \right)f^{n} \right\rangle = \langle v\cdot \nabla_x f^n\rangle + O(\Delta t) =  \langle v\cdot \nabla_x g^n\rangle + O(\Delta t),  
$$
where we used $\langle vM_\beta\rangle = 0$ and $\lambda=O(\Delta t)$ at fixed $\varepsilon$. 
Then, the last term in the first line of the scheme \eqref{mmscheme2} can be replaced by 
$\ep^{1-\alp}(1-\lambda)\left\langle v\cdot\nabla_x g^n\right\rangle$. 
Second, we will reformulate the second term of the first line of the scheme \eqref{mmscheme2} 
(which produces the fractional diffusion when $\ep$ goes to zero) in order 
to get the right limit when $\varepsilon\rightarrow 0$. 
To simplify the presentation, we present the reformulation using the Fourier variable
(because of the non-local character of the fractional laplacian in the space coordinates, the case in these coordinates is more delicate). 
The second term of (\ref{mmscheme2}) writes in the Fourier variable
\begin{align*}
\frac{\ep}{\ep^\alp}\lambda \left\langle \frac{\ii k\cdot v}{1+\ii \ep \lambda k\cdot v}M_\beta\right\rangle\fr^{n+1}(k)&=\frac{1}{\ep^\alp}\left\langle \frac{\ep^2\lambda^2 \left(k\cdot v\right)^2}{1+\ep^2\lambda^2(k\cdot v)^2} M_\beta\right\rangle \fr^{n+1}(k). 
\end{align*}
Then, as in (\ref{CalculDiffAN_2}) we make the change of variables $w=\ep\lambda |k| v$ 
in the velocity integration to get the following semi-discrete scheme (with $M_\beta(v)=\frac{m}{1+|v|^\beta}$)  
\begin{equation}
\label{mmscheme3}
\left\{
\begin{array}{l l l}
&\displaystyle\frac{\rho^{n+1}-\rho^n}{\dt}
&\displaystyle+
\lambda^\alp \mathcal{F}^{-1}\left(|k|^\alp
\left\langle 
\frac{(v\cdot e)^2}{1+(v\cdot e)^2}\frac{m}{(\ep\lambda |k|)^\beta+|v|^\beta}
\right\rangle \fr^{n+1}(k)
\right) \\
& &\displaystyle+\frac{\ep}{\ep^\alp+\dt}\left\langle v\cdot\nabla_x g^n\right\rangle=0,  \\
&\displaystyle\frac{g^{n+1}-g^n}{\dt}&\displaystyle+\frac{\ep}{\ep^\alp}v\cdot \nabla_x \rho^n M_\beta+\frac{\ep}{\ep^\alp}\left( v\cdot \nabla_x g^n-\left\langle v\cdot \nabla_x g^n \right\rangle M_\beta \right)=-\frac{1}{\ep^\alp}g^{n+1},
\end{array}
\right.
\end{equation}
where $e$ is any unitary vector and where we denoted by $\mathcal{F}^{-1}$ the inverse of the Fourier transform. The following proposition 
summarizes the main properties of this micro-macro scheme:

\begin{proposition}
\label{PropMicroMacroDiffAN}
In case \ref{HypDiffAN} (anomalous diffusion scaling), we introduce the following micro-macro the scheme defined for all $x\in \Omega$, $v\in \mathbb{R}^d$ and all time index $0\le n\le N, N\dt=T$ by
\[
\left\{
\begin{array}{l l l}
&\displaystyle\frac{\rho^{n+1}-\rho^n}{\dt}
&\displaystyle+
\lambda^\alp \mathcal{F}^{-1}\left(|k|^\alp
\left\langle 
\frac{(v\cdot e)^2}{1+(v\cdot e)^2}\frac{m}{(\ep\lambda |k|)^\beta+|v|^\beta}
\right\rangle \fr^{n+1}(k)
\right) \\
& &\displaystyle+\frac{\ep}{\ep^\alp+\dt}\left\langle v\cdot\nabla_x g^n\right\rangle=0 \\
\displaystyle
&\displaystyle\frac{g^{n+1}-g^n}{\dt}
&\displaystyle+\frac{\ep}{\ep^\alp}v\cdot \nabla_x \rho^n M_\beta+\frac{\ep}{\ep^\alp}\left( v\cdot \nabla_x g^n-\left\langle v\cdot \nabla_x g^n \right\rangle M_\beta \right)=-\frac{1}{\ep^\alp}g^{n+1},
\end{array}
\right.
\]
where $e$ is any unitary vector, and where initial conditions are $\rho^0(x)=\rho(0,x)$ and $g^0(x,v)=f(0,x,v)-\rho(0,x)M_\beta$. This scheme has the following properties:
\begin{romannum}
\item The scheme is of order $1$ for any fixed $\ep$ and preserves the total mass
\[
\forall n\in[\![1,N]\!], \fr^n(0)=\fr^0(0).
\]
\item The scheme is AP: 
for a fixed $\dt$, the scheme solves the diffusion equation when $\ep$ goes to zero 
	\[
\frac{\fr^{n+1}(k)-\fr^{n}(k)}{\Delta t}=-\kappa |k|^\alp \fr^{n+1}(k),
\]
where $\kappa$ is defined by (\ref{kappa}).
\end{romannum}

\end{proposition}

\begin{proof}
The proof of (i) is immediate, let us prove (ii). We have
when $\ep$ goes to zero with a fixed $\dt$, the scheme solves the anomalous diffusion equation. 
The equation on $g$ gives, when $\ep$ goes to zero. 
\[
g^{n+1}=O(\min(\ep,\ep^\alp)). 
\]
The equation on $\rho$ gives, in the Fourier variable 
\[
\frac{\fr^{n+1}-\fr^n}{\dt}+\lambda^\alp|k|^\alp\left\langle  \frac{(v\cdot e)^2}{1+(v\cdot e)^2}\frac{m}{(\ep\lambda |k|)^\beta+|v|^\beta}\right\rangle\fr^{n+1}
=O(\min(\ep^2,\ep^{1+\alp})). 
\]
Thanks to the following relations,  
$$
\lambda^\alp=\left(\frac{\dt}{\ep^\alp+\dt}\right)^\alp =1+O(\ep^\alp), \; \mbox{ and } \; \frac{m}{(\ep\lambda |k|)^\beta+|v|^\beta}=\frac{m}{|v|^\beta}+O(\ep^{\alp+d}),
$$
we get the implicit discretization of the anomalous diffusion equation (\ref{EulerDiffANFour}) 
when $\varepsilon\rightarrow 0$.
Hence the scheme solves the anomalous diffusion equation when $\ep$ goes to zero with a fixed $\dt$.
\qquad
\end{proof}

\subsection{Scheme based on an integral formulation of the equation}
\label{subsec:duhamelschemes}
In the previous parts, we wrote two AP schemes solving (\ref{KinDiff}) in the cases of the diffusion and anomalous diffusion limits, both of them were of order $1$ in time. Here we present a scheme based on a Duhamel formulation of (\ref{KinDiff}) which has uniform accuracy in $\ep$. 
This approach bears similarities with 
the UGKS scheme (see \cite{Xu, Mieussens}). As in the previous parts, the case of the classical 
diffusion gives easily an AP scheme but the large velocities effects require a specific treatment for the anomalous diffusion case. 

For both diffusion and anomalous diffusion, we start from (\ref{KinLim}) 
\[
\fr(t,k)=\left\langle \e^{-\frac{t}{\ep^\alp}\left( 1+\ii \ep k\cdot v \right)} \fr(0,k,v)\right\rangle +\int_0^{\frac{t}{\ep^\alp}}\e^{-s}\left\langle \e^{-\ii \ep s k\cdot v}M_\beta(v)\right\rangle \fr(t-\ep^\alp s,k)\dd s,
\]
hence evaluating at time $t=t_{n+1}$ leads to 
\begin{equation}
\label{Duhamel}
\fr(t_{n+1},k)=\hat{A}_0(t_{n+1},k) +\sum\limits_{j=0}^n \int_{\frac{t_{j}}{\ep^\alp}}^{\frac{t_{j+1}}{\ep^\alp}}\e^{-s}\left\langle \e^{-\ii \ep s k\cdot v}M_\beta\right\rangle \fr(t_{n+1}-\ep^\alp s,k)\dd s,
\end{equation}
where $\hat{A}_0$ is defined by (\ref{A0Four}).
We perform a quadrature of order $2$ in the integrals. Assuming that the time 
derivatives of $\fr$ are uniformly bounded in $\varepsilon$, we have: \newline$\forall j\in [\![1,N-1]\!]$, 
$\forall s\in \left[\frac{t_j}{\ep^\alp},\frac{t_{j+1}}{\ep^\alp}\right]$,
\begin{equation}
\label{quadrature}
\fr(t_{n+1}-\ep^\alp s,k)=a(\ep,s)\fr(t_{n+1}-t_j,k)+(1-a(\ep,s))\fr(t_{n+1}-t_{j+1},k)+O(\dt^2), 
\end{equation}
uniformly in $\varepsilon$, with $a(\ep,s)=1-\frac{\ep^\alp s-t_j}{\dt}$. 
Inserting \eqref{quadrature} in the integral term of \eqref{Duhamel} leads to 
\begin{align}
\int_{\frac{t_{j}}{\ep^\alp}}^{\frac{t_{j+1}}{\ep^\alp}}\e^{-s}\left\langle \e^{-\ii \ep s k\cdot v}M_\beta\right\rangle \fr(t_{n+1}-\ep^\alp s,k)\dd s&= c_j(k) \fr(t_{n+1}-t_{j},k) 
\nonumber
\\&
+b_j(k) \fr(t_{n+1}-t_{j+1},k) +O\!\left(\dt^2\right),  \label{quadrature2}
\end{align}
uniformly in $\varepsilon$, where we used the following notations $\forall j\in [\![0,N]\!]$
\begin{align}
\label{bj}
b_j(k)&=\int_{\frac{t_j}{\ep^\alp}}^{\frac{t_{j+1}}{\ep^\alp}}\frac{\ep^\alp s-t_j}{\dt}\e^{-s}\left\langle \e^{-\ii \ep s k\cdot v}M_\beta\right\rangle \dd s, \\
\label{cj}
c_j(k)&=\int_{\frac{t_j}{\ep^\alp}}^{\frac{t_{j+1}}{\ep^\alp}}\left( 1-\frac{\ep^\alp s-t_j}{\dt} \right)\e^{-s}\left\langle \e^{-\ii \ep s k\cdot v}M_\beta\right\rangle \dd s.
\end{align}

\subsubsection{AP scheme for diffusion}

We consider the case of the diffusion limit and we use the previous quadrature to write an AP scheme for the kinetic equation \eqref{KinDiff}.
We use the quadrature (\ref{quadrature2}) in (\ref{Duhamel}) to write
\begin{align}
\label{aascheme1}
\fr(t_{n+1},k)&=\hat{A}_0(t_{n+1},k)+\sum\limits_{j=0}^n\left( c_j \fr(t_{n+1-j},k) +b_j \fr(t_{n-j},k) +O\left(\dt^2\right)\right),
\end{align}
and as $n\dt\le T$ then $\sum\limits_{j=0}^n\dt^2=O(\dt)$ and we get a first order scheme that writes 
\[
\fr^{n+1}(k)=\frac{\displaystyle \hat{A}_0(t_{n+1},k)+\sum\limits_{j=1}^n \left( c_j \fr^{n+1-j}(k)+b_j\fr^{n-j}(k) \right) +b_0\fr^n(k)}{\displaystyle1-c_0}.
\]
To ensure the AP property, the time integration in the last two terms are computed exactly to get  
\begin{align}
\label{bj_diff}
b_j(k)&=\left\langle \left( \frac{\ep^\alp}{\dt} \frac{\e^{-\frac{t_j}{\ep^\alp} (1+\ep\ii k\cdot v)}\left( 1-\e^{-\frac{\dt}{\ep^\alp}(1+\ep\ii k\cdot v)} \right)}{ (1+\ep\ii k\cdot v)^2} 
-\frac{\e^{-\frac{t_{j+1}}{\ep^\alp}(1+\ep\ii k\cdot v)}}{1+\ep\ii k\cdot v}\right)M_\beta(v)\right\rangle, \\
\label{cj_diff}
c_j(k)&=\left\langle \left( 
\frac{\e^{-\frac{t_{j}}{\ep^\alp}(1+\ep\ii k\cdot v)}}{1+\ep\ii k\cdot v}
-
\frac{\ep^\alp}{\dt} \frac{\e^{-\frac{t_j}{\ep^\alp} (1+\ep\ii k\cdot v)}\left( 1-\e^{-\frac{\dt}{\ep^\alp}(1+\ep\ii k\cdot v)} \right)}{ (1+\ep\ii k\cdot v)^2} 
\right)M_\beta(v)\right\rangle.
\end{align}
We have the following proposition:

\begin{proposition}
\label{Propa1moinsaDiff}
In case \ref{HypDiff} (diffusion scaling), we consider the following the scheme defined for all $k$ and for all time index $0\le n\le N$ such that $N\dt= T$ by 
\[
\fr^{n+1}(k)=\frac{\displaystyle \hat{A}_0(t_{n+1},k)+\sum\limits_{j=1}^n \left( c_j \fr^{n+1-j}(k)+b_j\fr^{n-j}(k) \right) +b_0\fr^n(k)}{\displaystyle1-c_0},
\]
with the initial condition $\fr^0(k)=\fr(0,k)$ and $\hat{A}_0, b_j$ and $c_j$ defined in (\ref{A0Four})-(\ref{bj_diff})-(\ref{cj_diff}). 
This scheme has the following properties:
\begin{romannum}
\item The scheme is of order $1$ 
and the total mass is conserved 
\[
\forall n\in[\![1,N]\!], \fr^n(0)=\fr^0(0).
\]
\item The scheme is AP: 
for a fixed $\dt$, the scheme solves the diffusion equation when $\ep$ goes to zero 
	\[
\frac{\fr^{n+1}(k)-\fr^n(k)}{\Delta t}=-\left\langle (k\cdot v)^2 M_\beta(v)\right\rangle\fr^{n+1}(k).
\]
\end{romannum}
\end{proposition}

\begin{remark}
From the numerical test, the scheme appears to be of order
one uniformly in epsilon, as highlighted in Fig. \ref{a1moinsaDiff_Unif}.
This uniform property can be proved; however, since the proof is rather technical,
this is postponed to a work in preparation (see \cite{CrousHivertLemou2}). 
\end{remark}

\begin{proof}
From the construction of the scheme developed above, this scheme is clearly consistent and of order $1$ in time, 
uniformly in $\ep$. By induction we also remark that it conserves the numerical mass $\fr^n(0)$. 
Let us prove that, for a fixed $\dt$ the scheme solves the diffusion equation  when $\ep$ goes to zero. Indeed, for small $\ep$ the terms $b_j(k), c_j(k)$ for $j\ge 1$ and $\hat{A}_0$ are exponentially small, hence the scheme becomes 
\[
(1-c_0(k))\fr^{n+1}(k)= b_0(k)\fr^n(k)+o(\ep^\infty), \;\; \mbox{ when } \varepsilon \rightarrow 0. 
\]
 From the explicit expressions of $b_0(k)$ and $c_0(k)$ given in (\ref{bj_diff})-(\ref{cj_diff}) we get 
\begin{align*}
b_0(k)&=\frac{\ep^2}{\dt}+o(\ep^2), \\
c_0(k)&=1-\frac{\ep^2}{\dt}-\ep^2\left\langle (k\cdot v)^2 M_\beta(v)\right\rangle+o(\ep^2).
\end{align*}
Hence when $\ep$ goes to zero with a fixed $\dt$, the scheme of Prop. \ref{Propa1moinsaDiff} degenerates into
\[
\frac{\fr^{n+1}(k)-\fr^n(k)}{\dt}=-\left\langle (k\cdot v)^2 M_\beta(v)\right\rangle \fr^{n+1}(k),
\]
that is an implicit scheme solving the diffusion equation (\ref{Diff_eq}) in Fourier variable.

\qquad \end{proof}

\subsubsection{AP scheme for anomalous diffusion}

Now, we want to set an AP scheme for the kinetic equation (\ref{KinDiff}) in the anomalous diffusion scaling. 
Still using the notations defined in (\ref{A0Four})-(\ref{bj})-(\ref{cj}), the scheme in Prop. \ref{Propa1moinsaDiff} writes as 
\begin{align*}
(1-c_0) \fr^{n+1}(k)= \hat{A}_0(t_{n+1},k)+\sum\limits_{j=1}^n \left( c_j\fr^{n+1-j}(k)+b_j\fr^{n-j}(k) \right) +b_0\fr^n(k).
\end{align*}
 We remind that in the case of classical diffusion, the term $c_0 \fr^{n+1}$ produces 
 the right limit for small $\ep$. If we consider the same algorithm, as we are forced to consider a 
 bounded velocity space, the numerical second order moment of $M_\beta$ will be large but finite. 
 Therefore, we have to suitably transform the integrals $c_j$ and $b_j$ before discretizing in velocity. 
We have   
\begin{align*}
b_j &=\int_{\frac{t_j}{\ep^\alp}}^{\frac{t_{j+1}}{\ep^\alp}}\frac{\ep^\alp s-t_j}{\dt} \e^{-s}\left\langle \left(\e^{-\ii \ep s k\cdot v}-1\right) M_\beta\right\rangle\dd s
+\int_{\frac{t_j}{\ep^\alp}}^{\frac{t_{j+1}}{\ep^\alp}}\frac{\ep^\alp s-t_j}{\dt} \e^{-s}\dd s, \\
c_j &=\int_{\frac{t_j}{\ep^\alp}}^{\frac{t_{j+1}}{\ep^\alp}}\left( 1-\frac{\ep^\alp s-t_j}{\dt} \right)\e^{-s}\left\langle \left(\e^{-\ii \ep s k\cdot v}-1\right) M_\beta\right\rangle\dd s \\
&+\int_{\frac{t_j}{\ep^\alp}}^{\frac{t_{j+1}}{\ep^\alp}}\left( 1-\frac{\ep^\alp s-t_j}{\dt} \right)\e^{-s}\dd s, 
\end{align*} 
and in the velocity integrations we set $w=\ep|k|v$ as in (\ref{CalculDiffAN_1}), so when 
$M_\beta(v)=m/\left(1+|v|^\beta\right)$ 
we have 
\begin{eqnarray}
\lefteqn{b_j }~~& =&\ep^\alp |k|^\alp \left\langle \frac{m}{(\ep|k|)^\beta+|w|^\beta}\int_{\frac{t_j}{\ep^\alp}}^\frac{t_{j+1}}{\ep^\alp}   \frac{\ep^\alp s-t_j}{\dt} \e^{-s}\left( \e^{-\ii sw\cdot e}-1  \right)\dd s\right\rangle  \nonumber \\
& &+\int_{\frac{t_j}{\ep^\alp}}^{\frac{t_{j+1}}{\ep^\alp}}\frac{\ep^\alp s-t_j}{\dt} \e^{-s}\dd s , \nonumber  \\
\lefteqn{c_j }~~& =&\ep^\alp |k|^\alp \left\langle \frac{m}{(\ep|k|)^\beta+|w|^\beta}\int_{\frac{t_j}{\ep^\alp}}^\frac{t_{j+1}}{\ep^\alp}   \left( 1-\frac{\ep^\alp s-t_j}{\dt} \right)\e^{-s}\left( \e^{-\ii sw\cdot e}-1  \right)\dd s\right\rangle  \nonumber \\
& &+\int_{\frac{t_j}{\ep^\alp}}^{\frac{t_{j+1}}{\ep^\alp}}\left( 1-\frac{\ep^\alp s-t_j}{\dt} \right)\e^{-s}\dd s, \nonumber
\end{eqnarray}
with $e$ any unitary vector. To ensure the AP property, the time integrations are computed exactly to get
\begin{eqnarray}
\lefteqn{b_j }~~& =& 
 \e^{-\frac{t_j}{\ep^\alp}}\left(\frac{1-\e^{-\frac{\dt}{\ep^\alp}}}{\frac{\dt}{\ep^\alp}}-\e^{-\frac{\dt}{\ep^\alp}}\right)\left(1-\ep^\alp|k|^\alp \left\langle \frac{m}{(\ep|k|)^\beta+|v|^\beta} \right\rangle\right)
 \label{a1moinsa_b}  \\
 & &+ \ep^\alp |k|^\alp \left\langle \left( -\frac{\e^{-\frac{t_{j+1}}{\ep^\alp}(1+\ii v\cdot e)}}{1+\ii v\cdot e}
+\frac{\ep^\alp}{\dt} \e^{-\frac{t_j}{\ep^\alp}(1+\ii v\cdot e)}\frac{1-\e^{-\frac{\dt}{\ep^\alp}(1+\ii v\cdot e)}}{(1+\ii v\cdot e)^2}\right)  \frac{m}{(\ep|k|)^\beta+|v|^\beta}\right\rangle,  
 \nonumber \\
\lefteqn{c_j }~~& =& \e^{-\frac{t_j}{\ep^\alp}}\left(1-\frac{1-\e^{-\frac{\dt}{\ep^\alp}}}{\frac{\dt}{\ep^\alp}}\right)\left(1-\ep^\alp|k|^\alp \left\langle \frac{m}{(\ep|k|)^\beta+|v|^\beta} \right\rangle\right)
 \label{a1moinsa_c}  \\
& &+ \ep^\alp |k|^\alp \left\langle \left( \frac{\e^{-\frac{t_j}{\ep^\alp}(1+\ii v\cdot e)}}{1+\ii v\cdot e}
-\frac{\ep^\alp}{\dt} \e^{-\frac{t_j}{\ep^\alp}(1+\ii v\cdot e)}\frac{1-\e^{-\frac{\dt}{\ep^\alp}(1+\ii v\cdot e)}}{(1+\ii v\cdot e)^2}\right)  \frac{m}{(\ep|k|)^\beta+|v|^\beta}\right\rangle.   
 \nonumber
\end{eqnarray}

We have the following proposition

\begin{proposition}
\label{Propa1moinsaDiffAN}
In case \ref{HypDiffAN} (anomalous diffusion scaling), with the notations $\hat{A}_0$, $b_j$ and $c_j$ defined in (\ref{A0Four})-(\ref{a1moinsa_b})-(\ref{a1moinsa_c}), we consider  the following scheme defined for all $k$ and for all time index $0\le n\le N$ such that $N\dt = T$ by
\[
\fr^{n+1}(k)=\frac{ \displaystyle\hat{A}_0(t_{n+1},k)+\sum\limits_{j=1}^n \left( c_j\fr^{n+1-j}(k)+b_j\fr^{n-j}(k) \right) +b_0\fr^n(k) }{\displaystyle1-c_0},
\]
with the initial condition $ \fr^0(k)=\fr(0,k)$. This scheme has the following properties:
\begin{romannum}
\item The scheme is of order $1$ 
and preserves the total mass 
\[
\forall n\in[\![1,N]\!], \fr^n(0)=\fr^0(0).
\]
\item The scheme is AP: 
for a fixed $\dt$, the scheme solves the anomalous diffusion equation when $\ep$ goes to zero 
	\[
\frac{\fr^{n+1}(k)-\fr^{n}(k)}{\Delta t}=-\kappa |k|^\alp \fr^{n+1}(k),
\]
where $\kappa$ has been defined by (\ref{kappa}).
\end{romannum}
\end{proposition}

\begin{remark}
As in the case of diffusion, the numerical tests (see Fig. \ref{a1moinsaAN_Unif}) suggests that this scheme is order
one uniformly in epsilon. This uniform property can also be proved and this will be done in a future work 
 (see \cite{CrousHivertLemou2}). 
\end{remark}

\begin{proof}
As for the case of classical diffusion, the conservation of the mass is obtained 
by induction and the fact that it is of order $1$ comes  
from the Taylor expansion we performed in the integrals. Let us prove that for a fixed $\dt$ 
the scheme solves the anomalous diffusion equation when $\ep$ goes to zero. From the exponential decay 
of $A_0$ and $b_j,c_j, j\ge 1$ we have when $\ep$ goes to zero
\[
\left(1-c_0\right)\fr^{n+1}(k)=b_0(k)\fr^n(k)+o(\ep^\infty).
\]
On the one side we have
\[
b_0(k)=\frac{\ep^\alp}{\dt}+o(\ep^\alp),
\]
and on the other side we have
\begin{eqnarray*}
\lefteqn{c_0}~~&=& \ep^\alp |k|^\alp \left\langle  
\left( \frac{1}{1+\ii v\cdot e}-1\right)\frac{m}{(\ep|k|)^\beta+|v|^\beta} \right\rangle
+1-\frac{\ep^\alp}{\dt}+o(\ep^\alp)\\
&=&-\ep^\alp |k|^\alp \left\langle \frac{m}{|v|^\beta}\frac{(v\cdot e)^2}{1+(v\cdot e)^2} \right\rangle +1-\frac{\ep^\alp}{\dt}+o(\ep^\alp)\\
&=&-\ep^\alp |k|^\alp \kappa +1-\frac{\ep^\alp}{\dt}+o(\ep^\alp).
\end{eqnarray*}
We deduce that when $\ep$ goes to zero, the scheme degenerates into
\[
\frac{\fr^{n+1}(k)-\fr^n(k)}{\dt}=-\kappa |k|^\alp \fr^{n+1}(k),
\]
where $\kappa$ is defined by (\ref{kappa}). This is an implicit scheme for the anomalous diffusion equation (\ref{DiffAN_Eq_Four}).
\qquad \end{proof}

\begin{remark}
The previous schemes of  Prop. \ref{Propa1moinsaDiff}-\ref{Propa1moinsaDiffAN} can be modified 
into a numerical scheme which degenerates into a second order time approximation of the 
asymptotic model. Indeed, we can modify \eqref{aascheme1}
to get the following scheme  
\begin{align*}
\fr^{n+1}(k)&=A_0(t_{n+1},k)+b_0\fr^{n}(k)+\sum\limits_{j=1}^n\left( c_j \fr^{n+1-j}(k)+b_j\fr^{n-j}(k) \right)
\\&+(c_0+b_0-1)\frac{\fr^{n+1}(k)+\fr^n(k)}{2}+(1-b_0) \fr(t_{n+1},k).
\end{align*}
By construction, this scheme degenerates when $\ep$ goes to zero, into a Crank-Nicolson numerical scheme 
for the diffusion or anomalous diffusion equations. 
\end{remark}

\section{Numerical results}
\label{section4}

In this part, we present the numerical tests for the implicit scheme and the micro-macro scheme 
in the case of anomalous diffusion and for the scheme based on the Duhamel formulation of the 
kinetic equation in both cases of diffusion and of anomalous diffusion. In the sequel, 
we will denote by ISD (resp. ISA) the implicit schemes set in Prop. \ref{PropFullImplicitDiff} 
(resp. in Prop. \ref{PropFullImplicitDiffAN}) in the diffusion case 
(resp. anomalous diffusion case);  we will call MMSD (resp. MMSA) the micro-macro schemes in 
Prop. \ref{PropMicroMacroDiff} (resp. in Prop. \ref{PropMicroMacroDiffAN}) 
in the diffusion case (resp. anomalous diffusion case). Eventually, the acronyms DSD and DSA will refer to the schemes based on Duhamel Formulation in the Prop. \ref{Propa1moinsaDiff} 
and \ref{Propa1moinsaDiffAN}. 
Finally, DS and ADS denote the numerical schemes (\ref{EulerDiffFour}) and (\ref{EulerDiffANFour}) 
of the asymptotic models (diffusion and anomalous diffusion equations). 

In this part, we will consider $t>0, x\in [-1, 1], v\in \mathbb{R}$ and  the following initial data 
\[
f_0(x,v)=(1+\sin(\pi x))M_\beta(v),
\]
and periodic boundary conditions are imposed in $x$. In the diffusion case, we will consider the equilibrium 
\[
M_\beta(v)=\frac{1}{\sqrt{2\pi}}\e^{-v^2/2},
\]
whereas in the anomalous diffusion case, we will take 
\[
M_\beta(v)=\frac{m}{1+|v|^\beta}, 
\]
with $m$ chosen such that $\int_{\mathbb{R}}M_\beta(v)\dd v=1$.
In the sequel, unless other discretizations are mentioned, the following numerical parameters will hold:
\begin{romannum}
\item time: the final time will be set $T=0.1$ and $\dt=10^{-3}$. 
\item space: we will consider a uniform mesh in $x$ considering $N_x=64$ points. 
\item velocity: in both diffusion and anomalous diffusion case, a uniform velocity grid of the truncated 
domain $[-v_{\text{max}}, v_{\text{max}}]$ is considered, with $N_v$ points. 
In the diffusion case, we will take $v_{\text{max}}=10$ and $N_v=20$. 
In the anomalous diffusion case, we will consider $v_{\text{max}}=50$ and $N_v=200$.
The discretization ensures at the discrete level that $\int_{\mathbb{R}} vM_\beta(v) \dd v=0$. 
\end{romannum}
In the tests for the anomalous diffusion, we will consider $\alpha=1.5$. 

In the sequel, we will be interested in the relative error. The relative error is defined 
by the difference between a reference solution (obtained by a scheme using a refined mesh) 
and the solution obtained by the tested scheme through the following formula 
\begin{equation}
\label{error}
\text{Error}(T)=\frac{\| \rho_{\text{reference}}(T)-\rho_{\text{scheme}}(T) \|_2}{\| \rho_{\text{reference}}(T) \|_2},
\end{equation}
where $\|\cdot \|_2$ denotes the discrete $L^2$ norm and the reference scheme will be precised in each case.

\subsection{The implicit scheme in the case of the anomalous diffusion limit}

In this section, we test the  ISA scheme. 
We check that it is of order $1$ for a fixed value of $\varepsilon$ and the AP property  when $\ep$ goes to $0$ 
for a fixed value of $\Delta t$. 

Firstly, we remark that the ISD scheme used for the diffusion case does not preserve 
the asymptotic of anomalous diffusion: the left hand side of Fig. \ref{CinetiqueAN_AP} shows that 
for $\ep=10^{-6}$ the implicit scheme returns a result very different from the result given by (\ref{EulerDiffANFour}). 
Indeed, when $\varepsilon <\!\!<1$, the ISD scheme produces a stationary solution. 
So it appears that the change of variables in the velocity integrals is necessary to capture 
the anomalous diffusion limit, as shown in the right hand side of Fig. \ref{CinetiqueAN_AP}.

\begin{figure}[!htbp]
\begin{center}
\begin{tabular}{@{}c@{}c@{}}
\includegraphics[width=6.5cm]{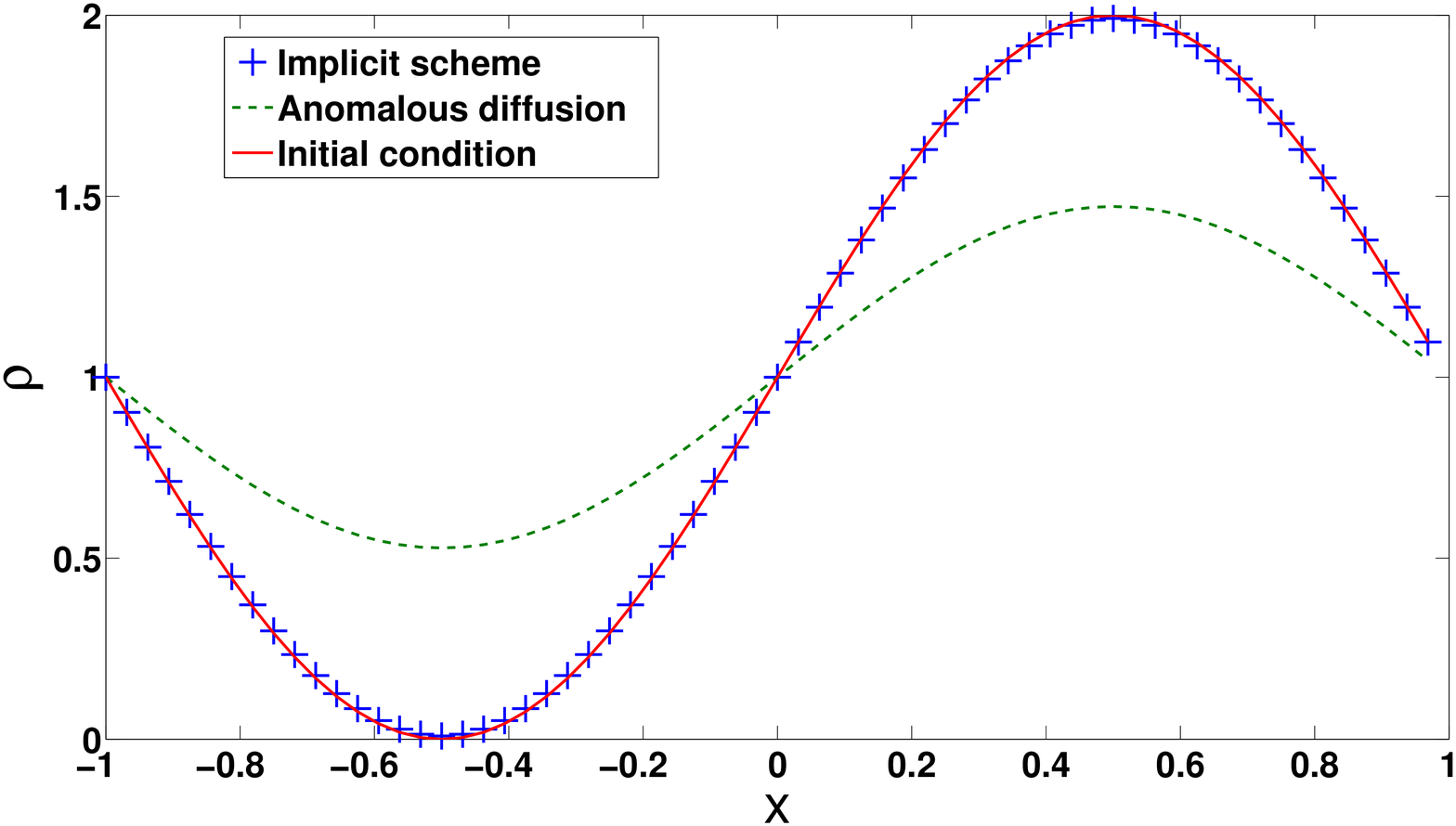} &     
\includegraphics[width=6.5cm]{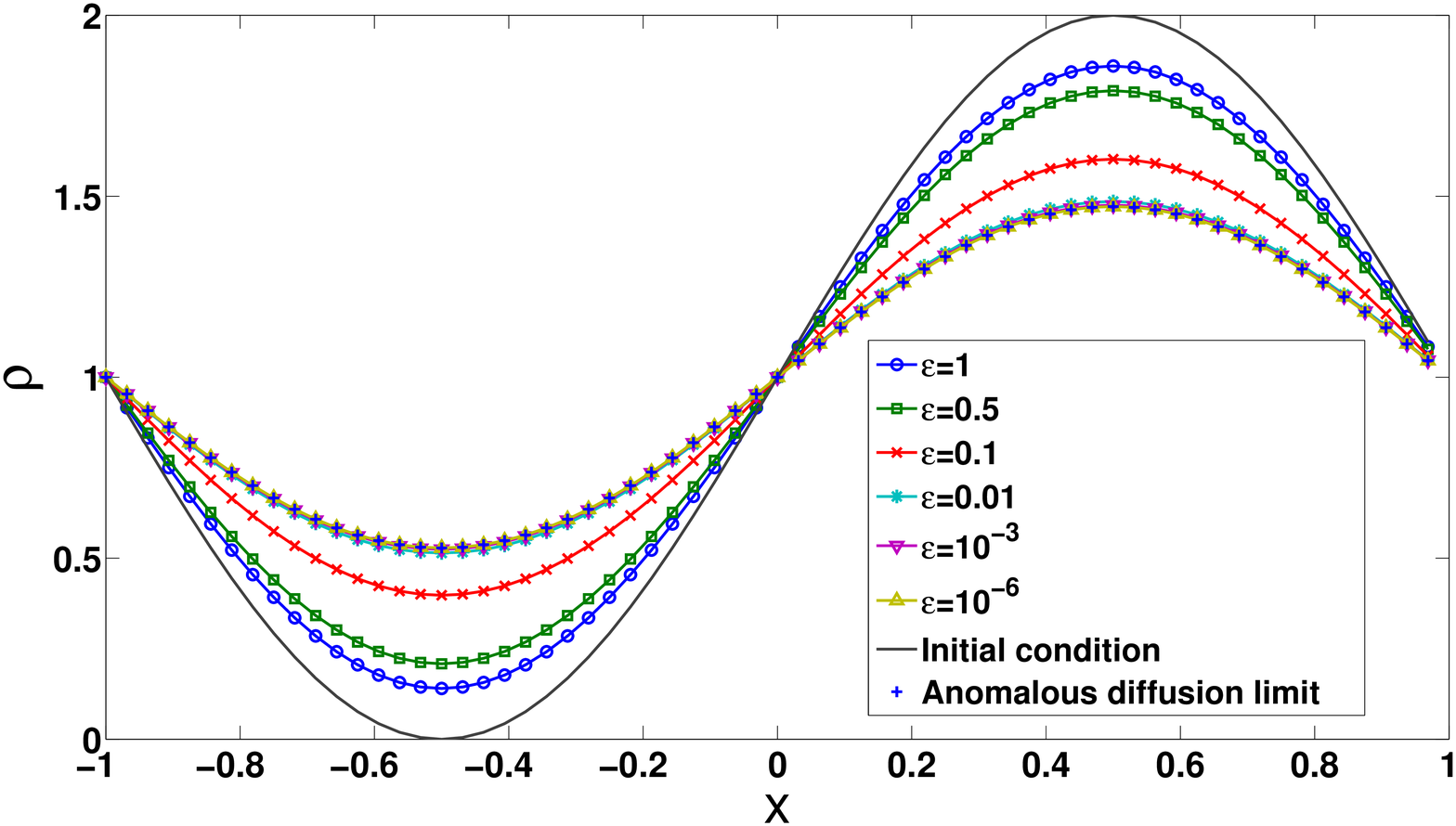}
\end{tabular}
\caption{For $\dt=10^{-3}$, the densities $\rho(t=0.1, x)$. Left: ISD scheme (with $\ep=10^{-6}$), 
by the ADS scheme and the initial data. Right: ISA scheme for different values of $\ep$ and the ADS scheme.}
\label{CinetiqueAN_AP}
\end{center}
\end{figure}



The left hand side of Fig. \ref{CinetiqueAN_46} shows the error  between the reference scheme 
(defined as the ADS scheme) and the ISA scheme as a function of $\varepsilon$, for 
$\alpha=0.8, 1, 1.5$. We observe that the convergence of the kinetic equation to the anomalous diffusion 
when $\varepsilon$ goes to zero arises with a speed $\alp$. 
Theorically, the convergence to the anomalous diffusion equation arises with a speed $\ep^{\min(\alp, 2-\alp)}$, but as we consider a finite domain for $v$ in the tests, the convergence rate always appears to be $\ep^\alp$. This result will be proved in \cite{CrousHivertLemou2}.


\begin{figure}[!htbp]
\begin{center}
\begin{tabular}{@{}c@{}c@{}}
\includegraphics[width=6.5cm]{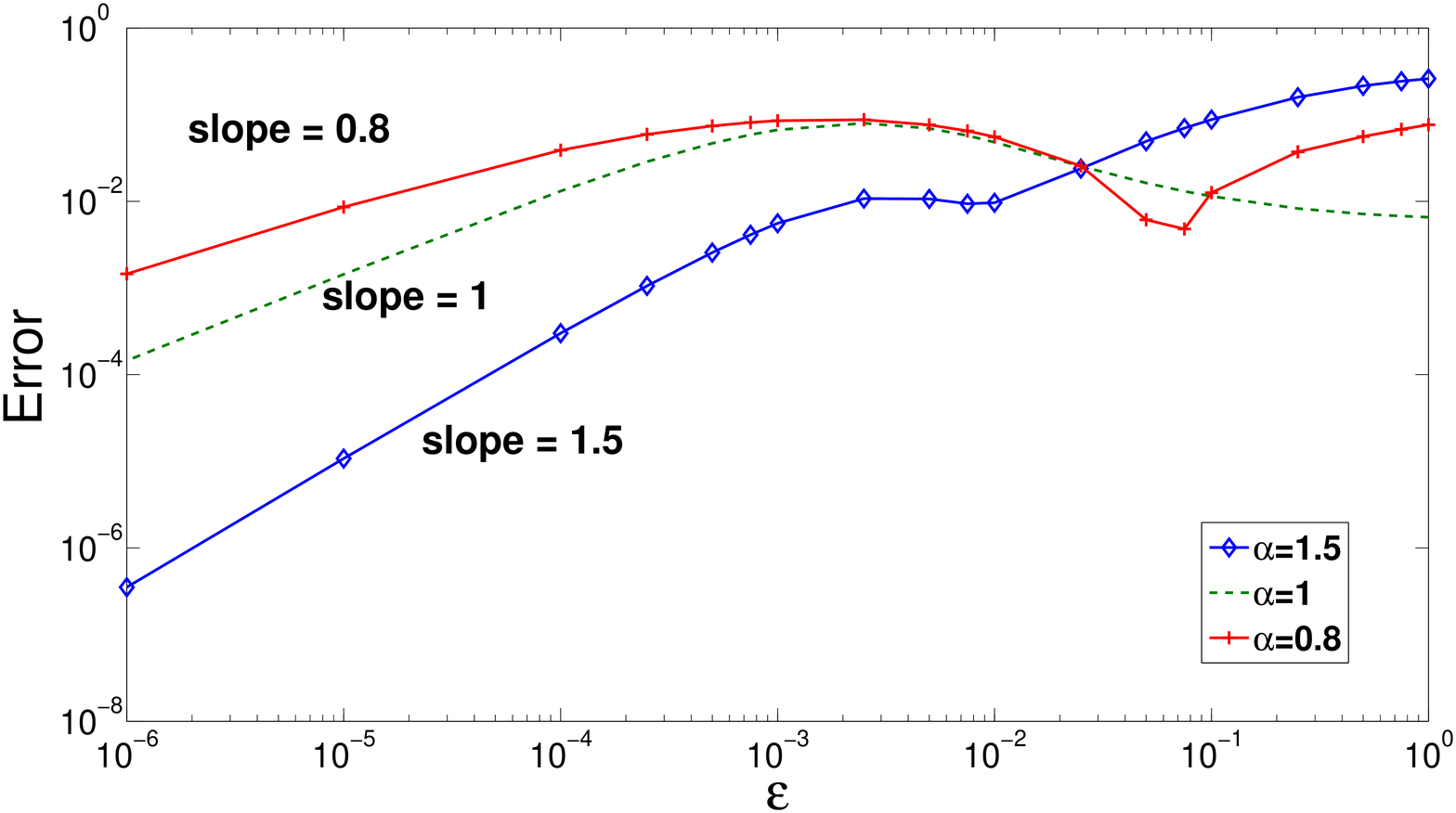} &     
\includegraphics[width=6.5cm]{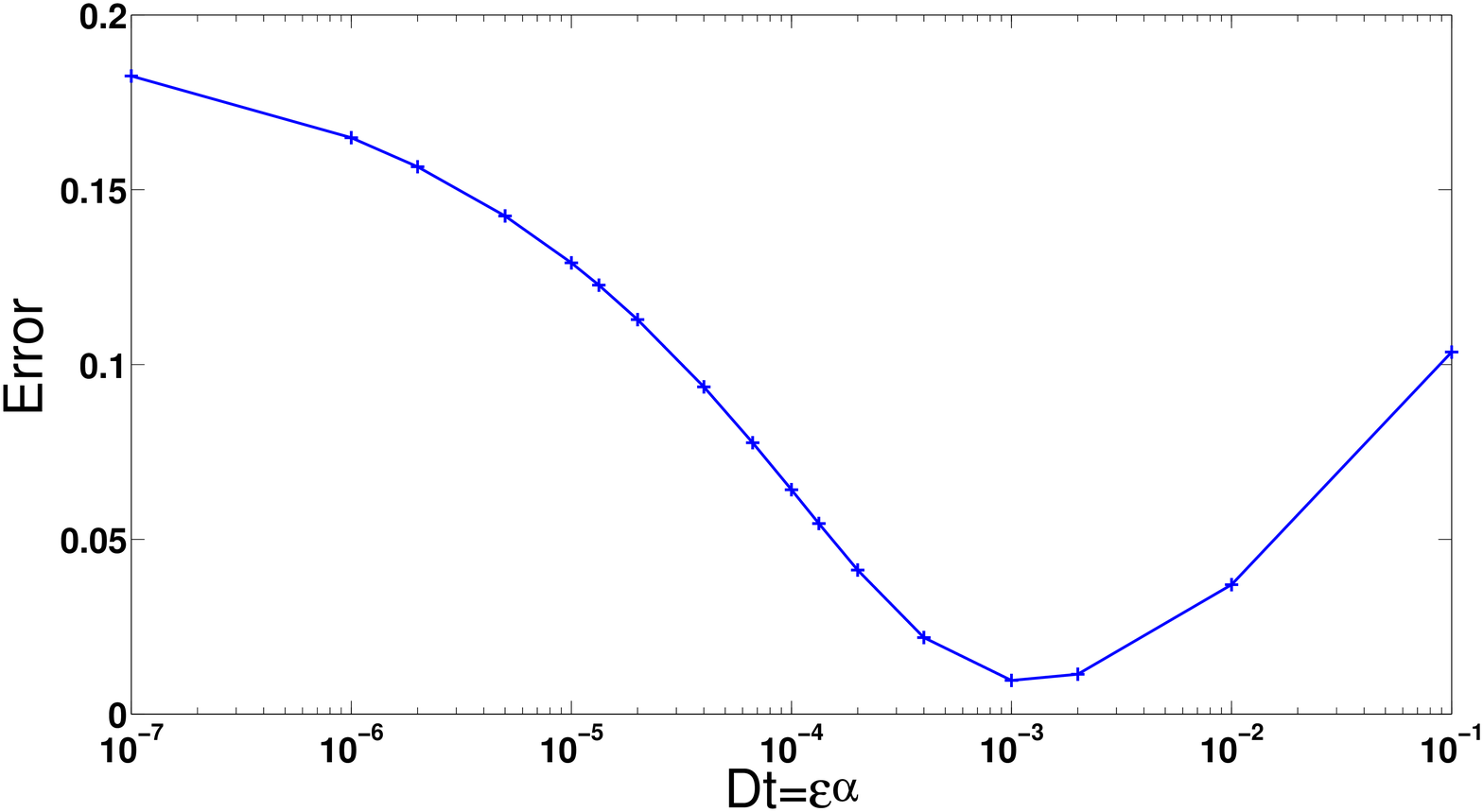}
\end{tabular}
\caption{Left: For $\dt=10^{-3}$ and $\alp=0.8, 1, 1.5$, the error as a function of $\varepsilon$ 
between the  reference ADS scheme and the ISA scheme (log scale). Right: The error between the ISA scheme for a range of $\dt=\ep^\alp$ and the ADS scheme 
as a function of $\dt$, which is fixed equal to $\varepsilon^\alpha$ (log scale). }
\label{CinetiqueAN_46}
\end{center}
\end{figure}



However the convergence in time of the  ISA scheme to the DSA scheme is not uniform with respect to $\varepsilon$.   
Setting $\dt=\ep^\alp$ in the scheme and letting $\dt$ go to $0$ shows that the densities does not converge to the density given by the anomalous diffusion equation. 
On the right hand side of Fig.  \ref{CinetiqueAN_46}, we plot the error in time between the ADS scheme as reference and the ISA scheme 
computed for a different time steps satisfying $\dt=\ep^\alp$.  
We observe that the error does not converge to $0$ when $\dt$ goes to $0$, illustrating the lack of uniformity.



\subsection{The micro-macro scheme}

As we saw in Section \ref{SectionMMscheme}, the MMSD scheme is known to be efficient in the diffusion case 
(see \cite{LemouMieussens}). In this part, we focus on the MMSA scheme in the case of the kinetic equation with anomalous diffusion limit. 
As it does not use only the Fourier variable, we must consider a grid for the space variable. Here we will take $x\in [-1,1]$ discretized with $N_x=32$ points. As we solve the transport of the micro part (equation on $g$) 
in an explicit way (using finite differences), a CFL condition has to be imposed; 
then the time step is chosen small enough. 

We start by testing the consistency of the scheme. We fix $\alp=1.5$ and for a decreasing range of time steps, 
we compute the solution given by the MMSA scheme. Then, for $\varepsilon=1, 0.5$, we compare it to the reference solution given by MMSD with $\Delta t=10^{-6}$. For $\varepsilon=10^{-7}, 10^{-8}$, we compare to the 
solution given by the ADS scheme.  
Fig. \ref{mmDiffAN_Cons2} shows the error curves obtained for different values of $\ep$. 
It appears that the MMSA scheme is of order $1$ in time.


\begin{figure}[!ht]
\centering
\includegraphics[width=13cm]{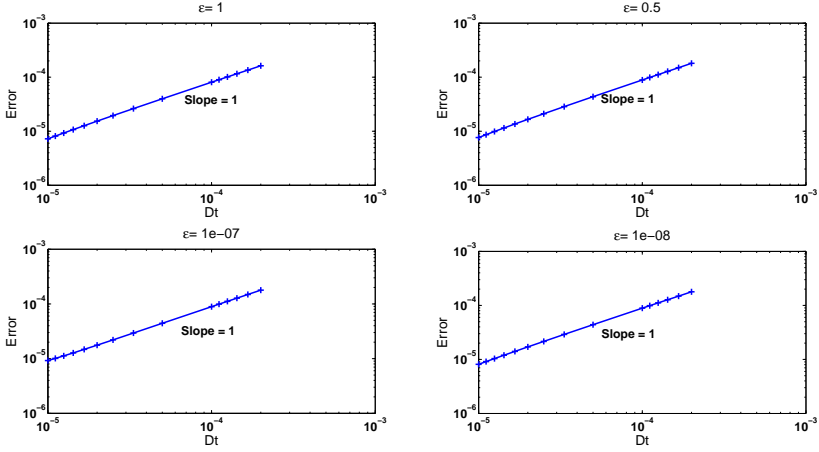}
\caption{For $\ep=1, 0.5, 10^{-7}, 10^{-8}$, the error in time between the   MMSA scheme computed for a range 
of $\dt$ and the reference scheme computed with  $\dt=10^{-6}$ (log scale). }
\label{mmDiffAN_Cons2}
\end{figure}

In order to show that the scheme preserves the asymptotic of anomalous diffusion we compute  
the densities $\rho(t=0.1, x)$ obtained by the MMSA scheme for a range of $\ep$; we compare them to the density 
obtained by the ADS scheme computed with the same $\dt$. In order to respect 
the CFL condition, we took $\dt=10^{-4}$ for $\ep=1, 0.5,0.1, 0.01$ and 
$\dt=10^{-3}$ for the smallest $\ep$.The left hand side of  Fig. \ref{mmDiffAN_AP} shows these densities in the 
case $\alp=1.5$, we observe that they are very close to the anomalous diffusion limit for small $\ep$. 
On the right hand side of Fig. \ref{mmDiffAN_AP}, the errors associated to this latter study are displayed 
in the cases $\alp=0.8, \alp=1$ and $\alp=1.5$, the ADS scheme being considered as reference. 
We observe that the convergence happens with speed $\alp$, as expected.

\begin{figure}[!htbp]
\begin{center}
\begin{tabular}{@{}c@{}c@{}}
\includegraphics[width=6.5cm]{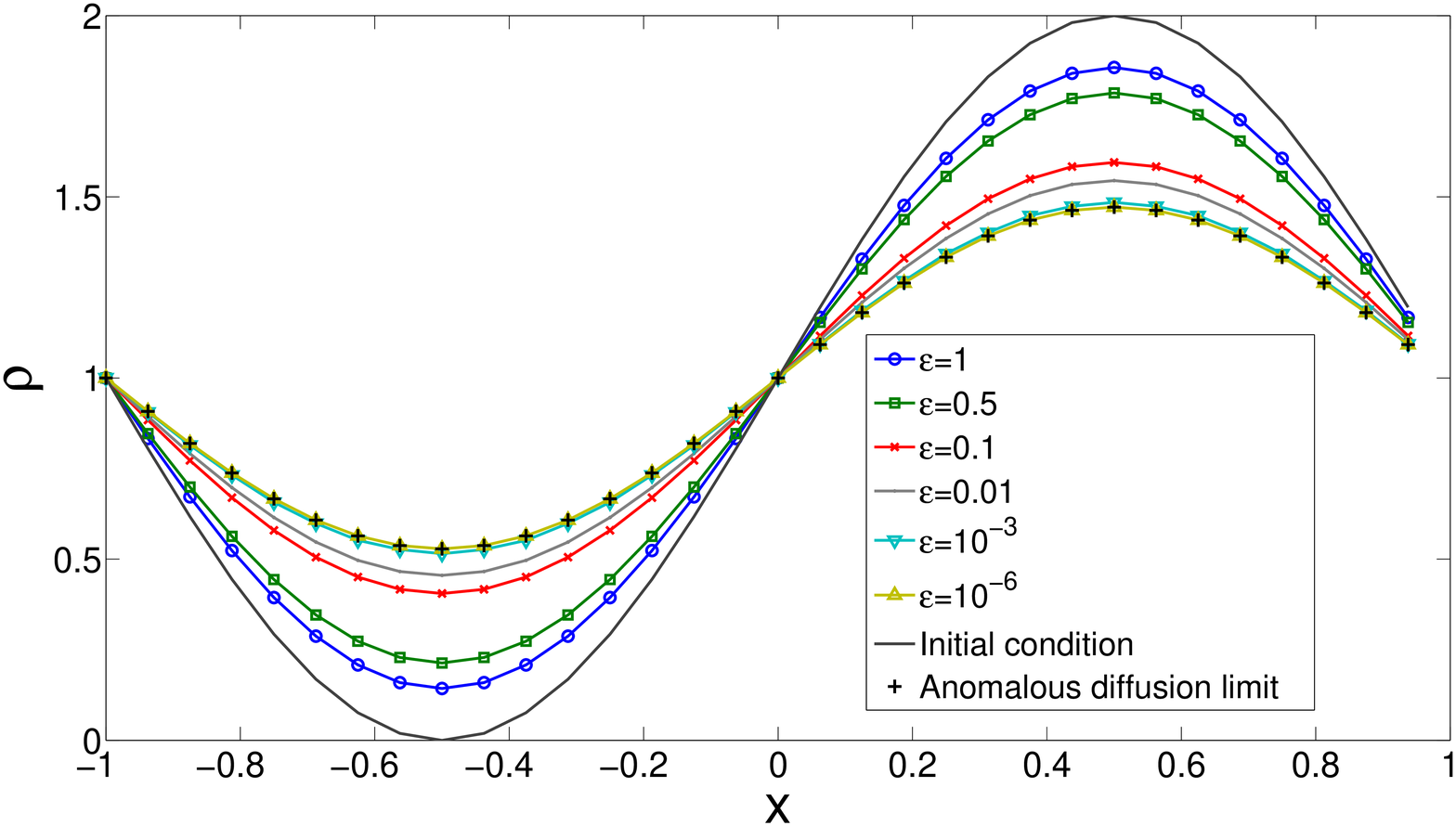} &     
\includegraphics[width=6.5cm]{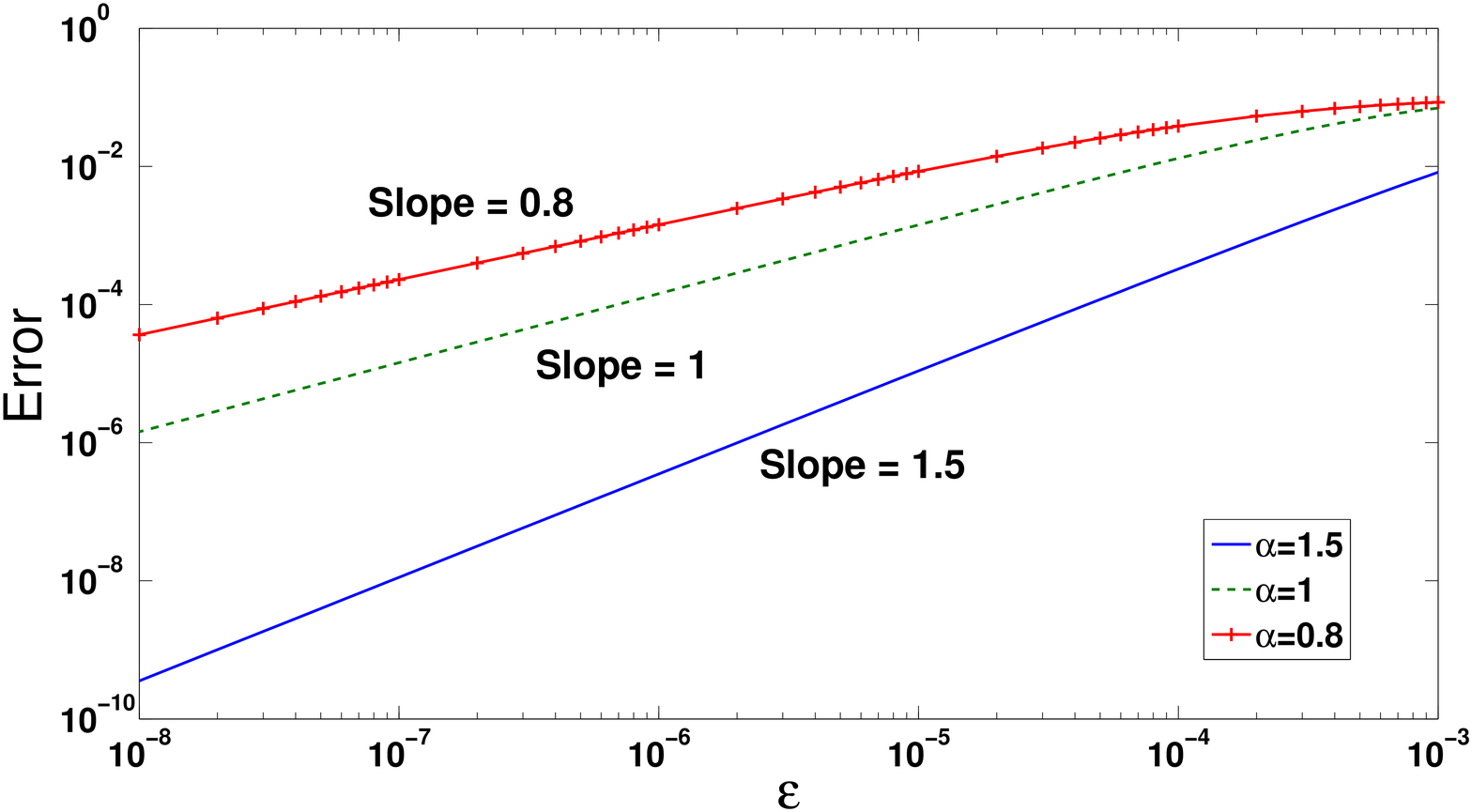}
\end{tabular}
\caption{The densities $\rho(t=0.1, x)$. Left: MMSA scheme for different values of $\ep$ 
and the DS scheme. Right: For $\dt=10^{-3}$ and $\alpha=0.8, 1, 1.5$,  the error with respect to $\varepsilon$ 
between the reference ADS scheme and the  MMSA scheme (log scale).}
\label{mmDiffAN_AP}
\end{center}
\end{figure}

%
%
%

\subsection{The integral formulation based scheme}

In this section, we focus on the  Duhamel formulation based schemes DSD and DSA in both diffusion and anomalous diffusion regimes. We put their properties of consistency, AP character and uniformity with respect to $\ep$, written in Prop.  \ref{Propa1moinsaDiff} and Prop. \ref{Propa1moinsaDiffAN}, in evidence in these two cases.
 
 
\subsubsection{Case of diffusion limit}

In the case of diffusion limit, we compare the results obtained by the  DSD scheme 
to the results obtained with the  ISD scheme. The first thing we check is 
the convergence of the algorithm when $\dt$ goes to zero. 
We choose a range of $\ep$ and we compare the densities obtained with the two algorithms as a function of $\dt$. 
Fig. \ref{a1moinsaDiff_Cons6}  displays the errors associated to these comparisons. 
It appears that for $\ep \sim 1$, the scheme is of order $2$ and of order $1$ for small $\ep$.


\begin{figure}[!ht]
\centering
\includegraphics[width=8cm]{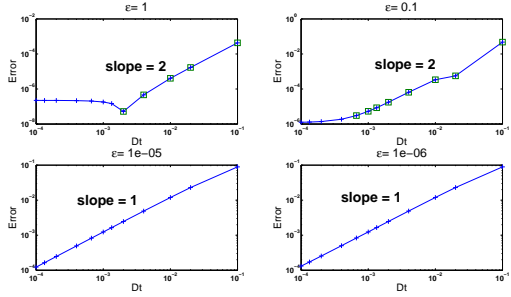}
\caption{For $\ep=1, 0.1, 10^{-5}, 10^{-6}$, the error in time between the DSD scheme 
and the reference ISD scheme computed for  $\dt=10^{-6}$ (log scale).}
\label{a1moinsaDiff_Cons6}
\end{figure}

In order to show the AP character of the scheme, we compute the densities given 
by  the scheme DSD for a decreasing range of $\ep$ ($\Delta t$ being fixed) and we check that they converge 
to the solution of the diffusion equation as $\ep$ goes to $0$. 
The densities obtained are presented in Fig. \ref{a1moinsaDiff_AP1}.  

\begin{figure}[!ht]
\centering
\includegraphics[width=8cm]{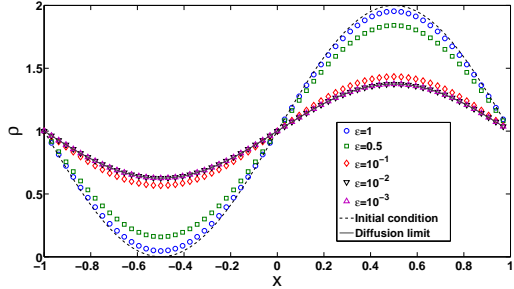}
\caption{For $\dt=10^{-3}$ the densities $\rho(t=0.1, x)$ given by the DSD scheme for different values of 
$\ep$ and the DS scheme. }
\label{a1moinsaDiff_AP1}
\end{figure}

To highlight the fact that the scheme is first order in time uniformly in $\ep$, we compare the results given by the ISD scheme with $\dt=10^{-8}$ to the results given by the DSD scheme for a range of $\dt$. These errors are displayed in the left hand side of Fig. \ref{a1moinsaDiff_Unif} as a function of $\ep$ where we observe that the error curves are stratified with respect to $\dt$, showing the first order uniform accuracy of the scheme. This property appears also in the right hand side of Fig. \ref{a1moinsaDiff_Unif} where we plotted the error between the DSD scheme computed for $\dt=\ep^2$  and the limit DS scheme computed for $\dt=10^{-7}$. This error tends to zero when $\dt$ and $\ep$ go to zero, thanks to the uniformity of the DSD scheme.

\begin{figure}[!htbp]
\begin{center}
\begin{tabular}{@{}c@{}c@{}}
\includegraphics[width=6.5cm]{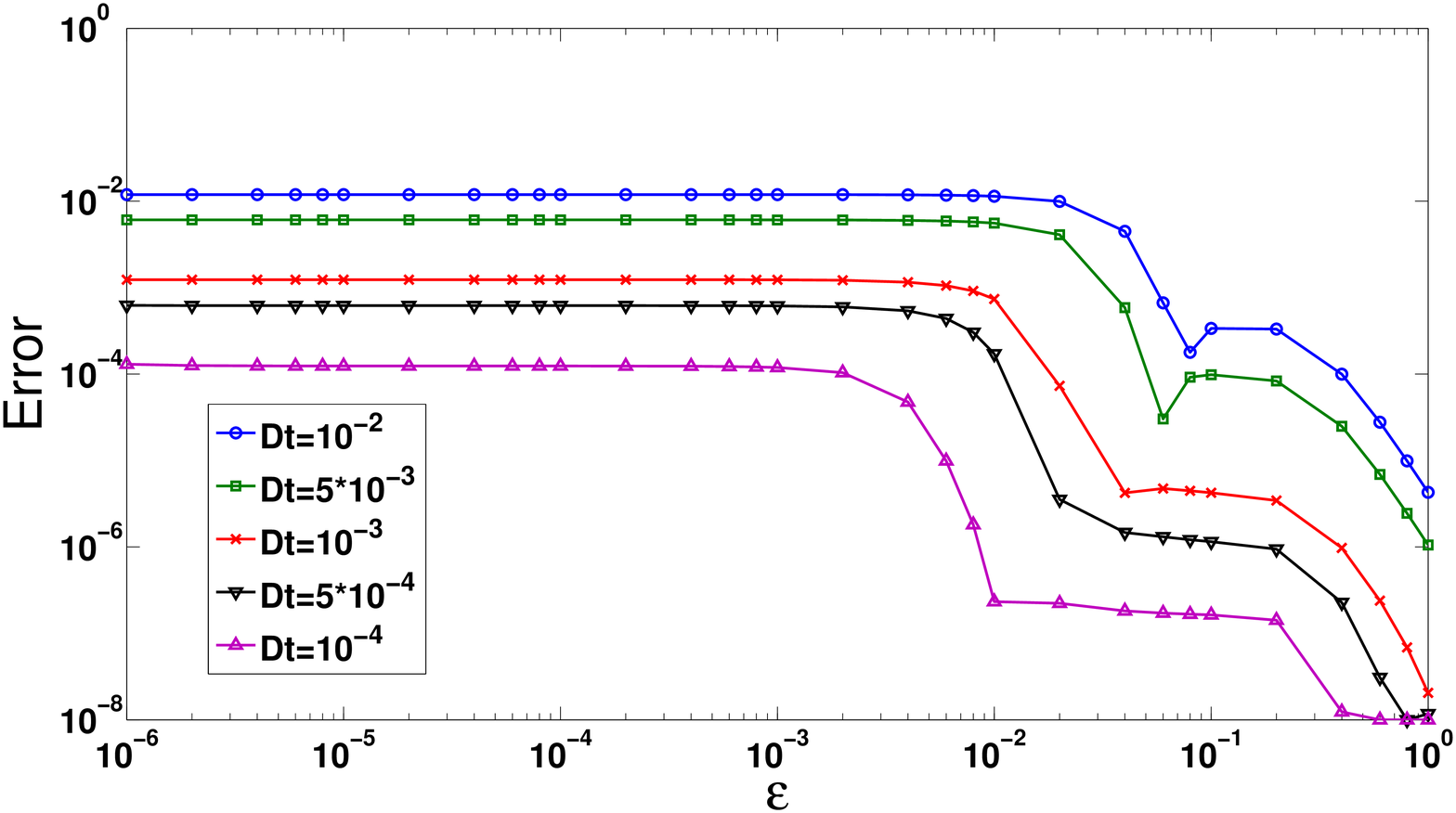} &     
\includegraphics[width=6.5cm]{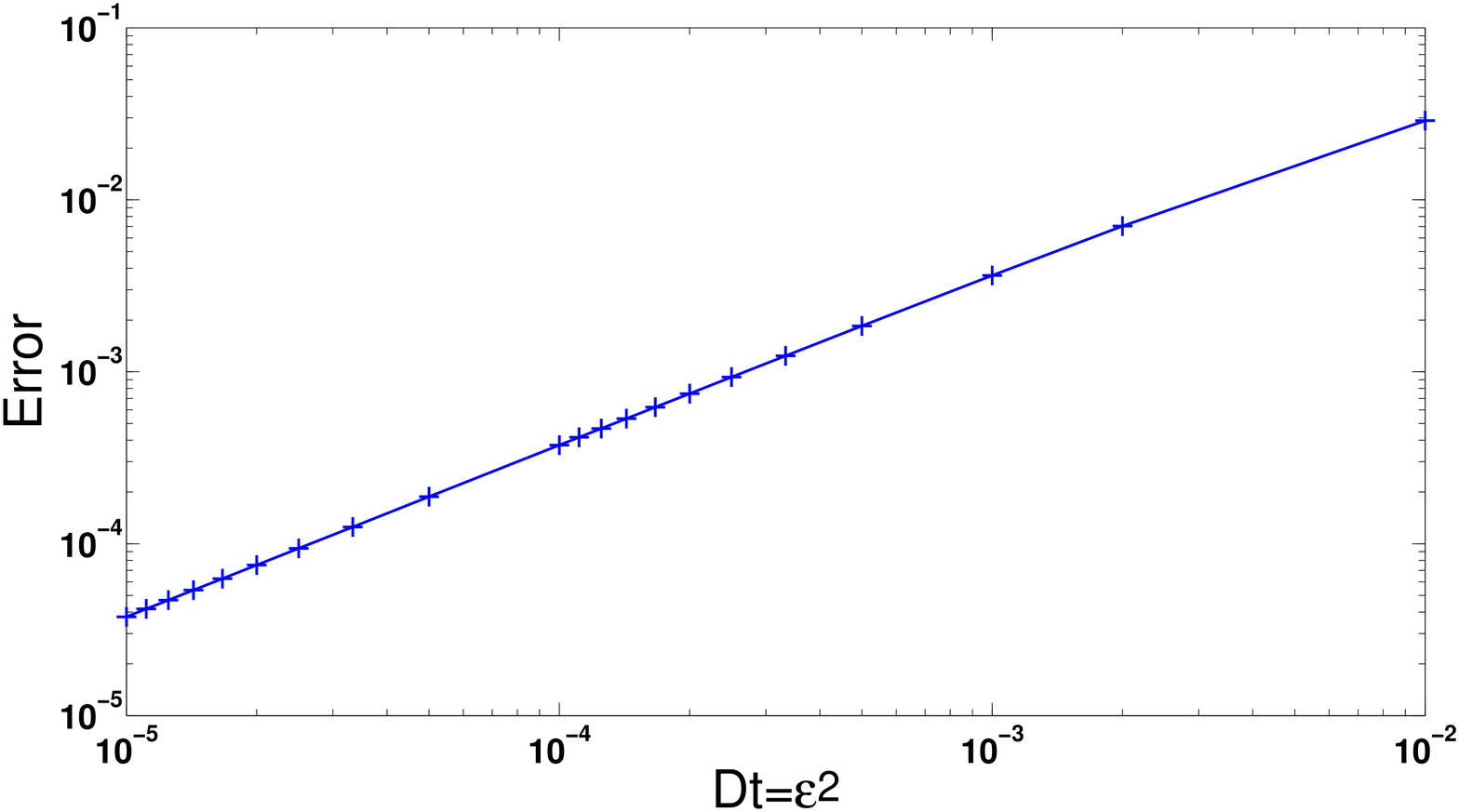}
\end{tabular}
\caption{Left: The difference between the ISD scheme computed for $\dt=10^{-8}$  and the DSD scheme computed for a range of $\dt$ as a function of $\ep$ (log scale). Right: The difference between the DSD scheme computed for $\dt=\ep^\alp$  and the DS scheme computed for $\dt=10^{-7}$ as a function of $\dt$ (log scale).}
\label{a1moinsaDiff_Unif}
\end{center}
\end{figure}

%
%
%
%

\subsubsection{Case of anomalous diffusion limit}

In this section we test the DSA scheme. 
In the case of anomalous diffusion limit, we compare the results given by the  DSA scheme to the results given by the DSA scheme computed with a smaller time step, to make the same properties appear. 
We start by highlighting the consistency of the scheme. Considering different values of $\ep$, 
we compute a reference solution with the scheme DSA with $\dt=10^{-5}$ and we compare it to the 
results given by the DSA scheme for some larger time steps. Fig. 
\ref{a1moinsaAN_Cons6} show 
error study of the convergence of the densities obtained. We observe that for $\ep=1, 0.5$ the scheme seems to be of order $2$ 
whereas it is of order $1$ for smaller values of $\ep$.

\begin{figure}[!ht]
\centering
\includegraphics[width=8cm]{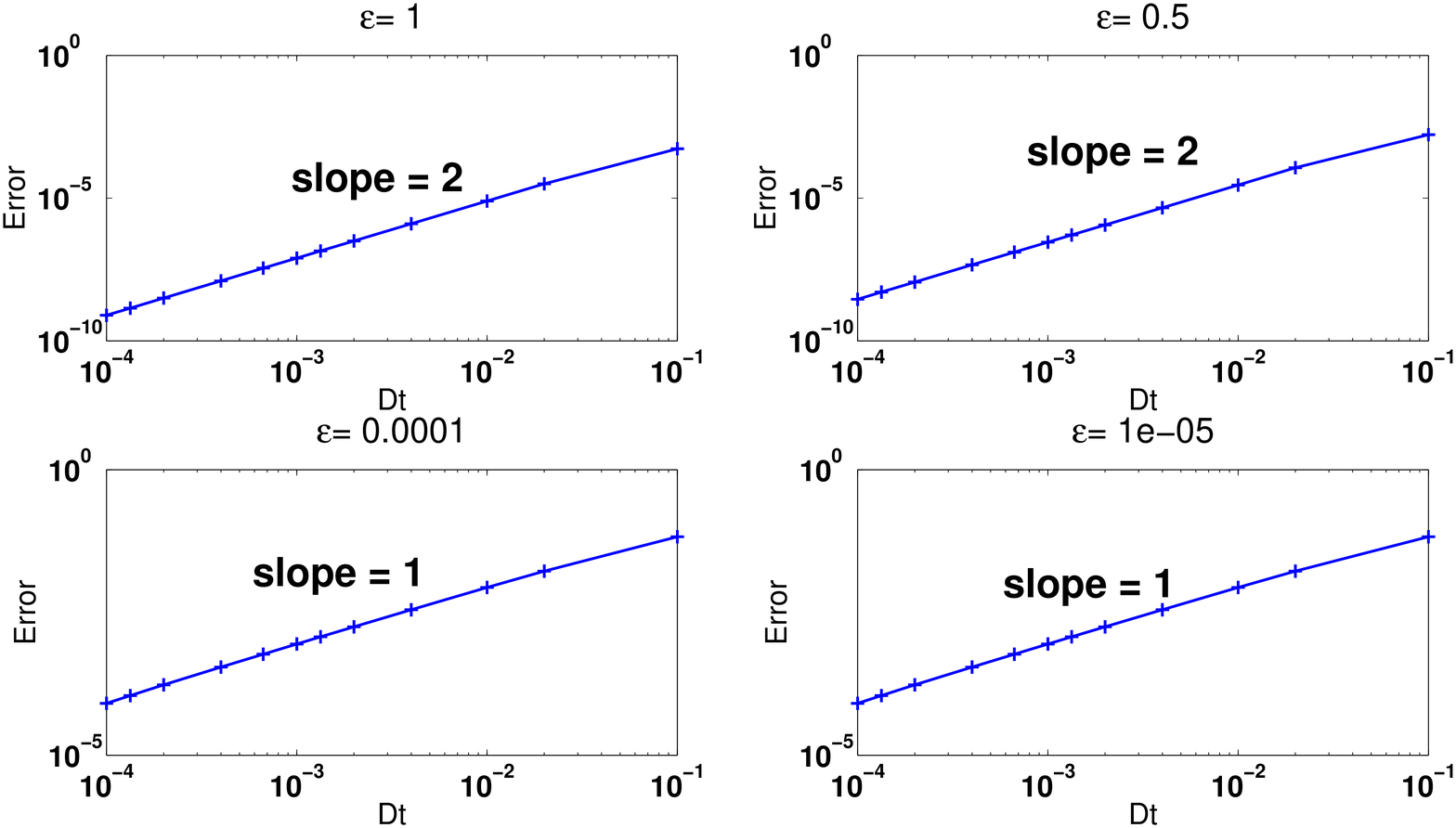}
\caption{For  $\ep=1, 0.1, 10^{-4}, 10^{-5}$, the error as a function of $\Delta t$ 
between the DSA scheme and the DSA scheme computed for  $\dt=10^{-5}$ (log scale). }
\label{a1moinsaAN_Cons6}
\end{figure}

Then, we can study the AP character of the DSA scheme. The time step being fixed to $\dt=10^{-3}$, 
we check that the results given by the DSA scheme converge to the result given by the ADS scheme 
when $\ep$ goes to $0$. Fig. \ref{a1moinsaAN_AP} presents the densities $\rho(t=0.1, x)$ 
obtained by these two schemes for different values of $\varepsilon$
and then, the error as a function of $\varepsilon$ is plotted for $\alpha=0.8, 1, 1.5$. As previously, 
the expected numerical speed of the convergence to the anomalous diffusion is recovered.


\begin{figure}[!htbp]
\begin{center}
\begin{tabular}{@{}c@{}c@{}}
\includegraphics[width=6.5cm]{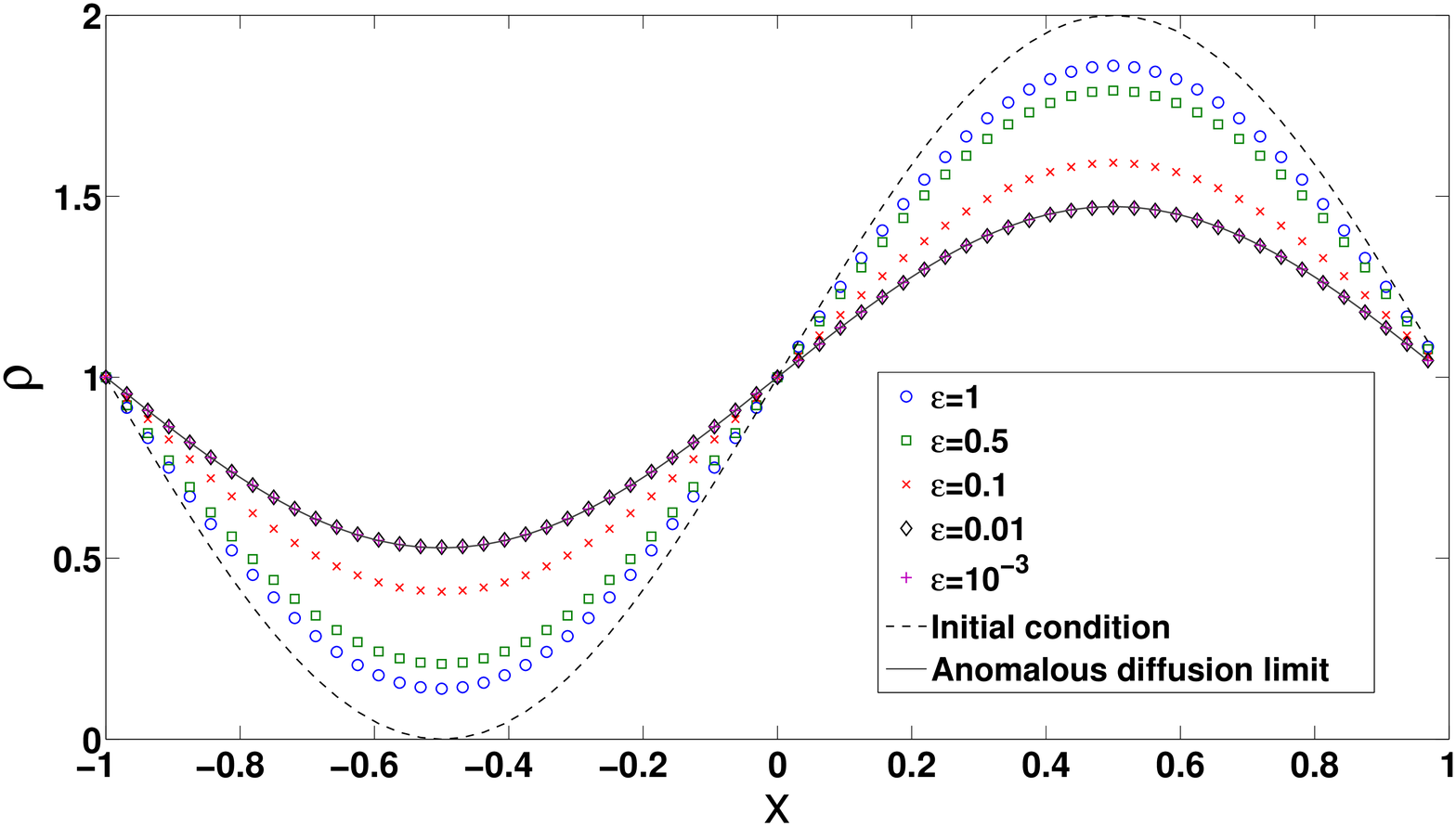} &     
\includegraphics[width=6.5cm]{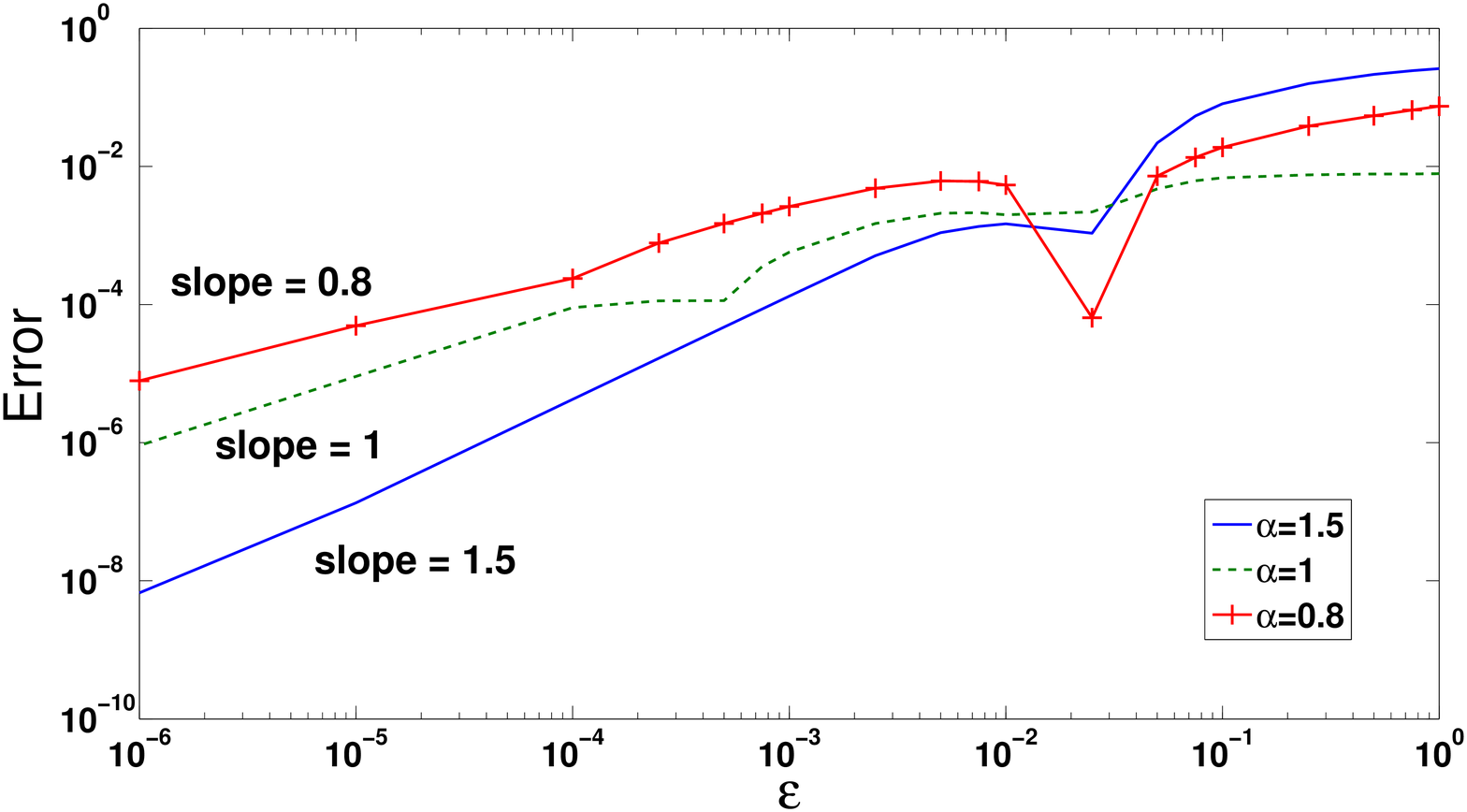}
\end{tabular}
\caption{Left: For $\dt=10^{-3}$ the densities $\rho(t=0.1, x)$ given by the DSA scheme for different values 
of $\ep$ and the ADS scheme. Right: For $\dt=10^{-3}$ and $\alpha=0.8, 1, 1.5$, the error in $\ep$ between the reference ADS scheme and the DSA scheme (log scale). }
\label{a1moinsaAN_AP}
\end{center}
\end{figure}

%

To highlight the fact that the scheme is first order in time uniformly in $\ep$, we compare the results given by the DSA scheme with $\dt=10^{-5}$ to the results given by the same DSA scheme for a range of $\dt$. These errors are displayed in the left hand side os Fig. \ref{a1moinsaAN_Unif} as a function of $\ep$ where we observe that the error curves are stratified with respect to $\dt$, showing the uniformity of the scheme with respect to $\varepsilon$. 
As for the case of diffusion limit, the right hand side of Fig. \ref{a1moinsaAN_Unif} presents the error between the DSA scheme computed for $\dt=\ep^\alp$ and the limit ADS scheme computed for  $\dt=10^{-7}$. Since the DSA scheme is of order $1$ uniformly in $\ep$, this error tends to zero when $\dt$ and $\ep$ go to zero.

\begin{figure}[!htbp]
\begin{center}
\begin{tabular}{@{}c@{}c@{}}
\includegraphics[width=6.5cm]{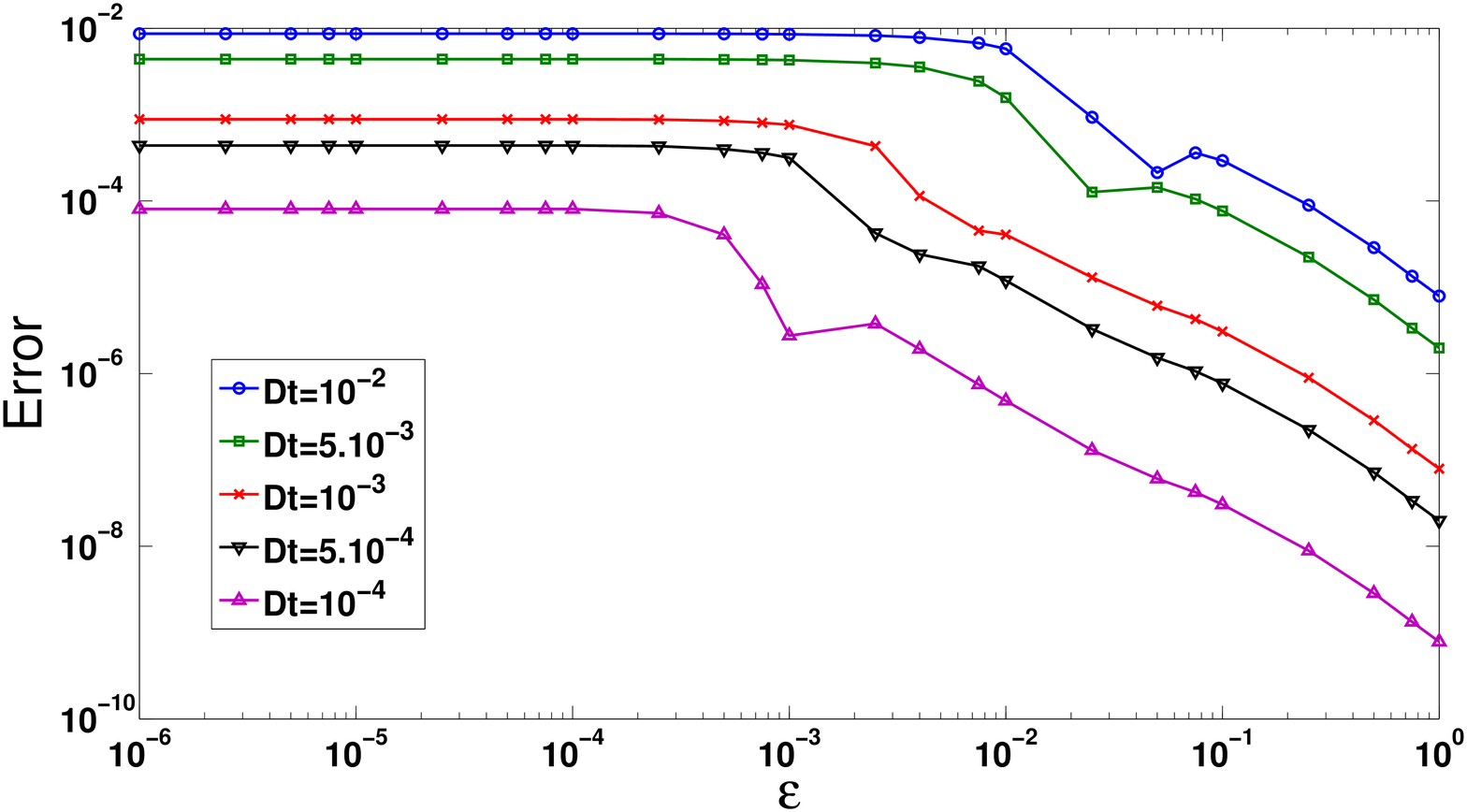} &     
\includegraphics[width=6.5cm]{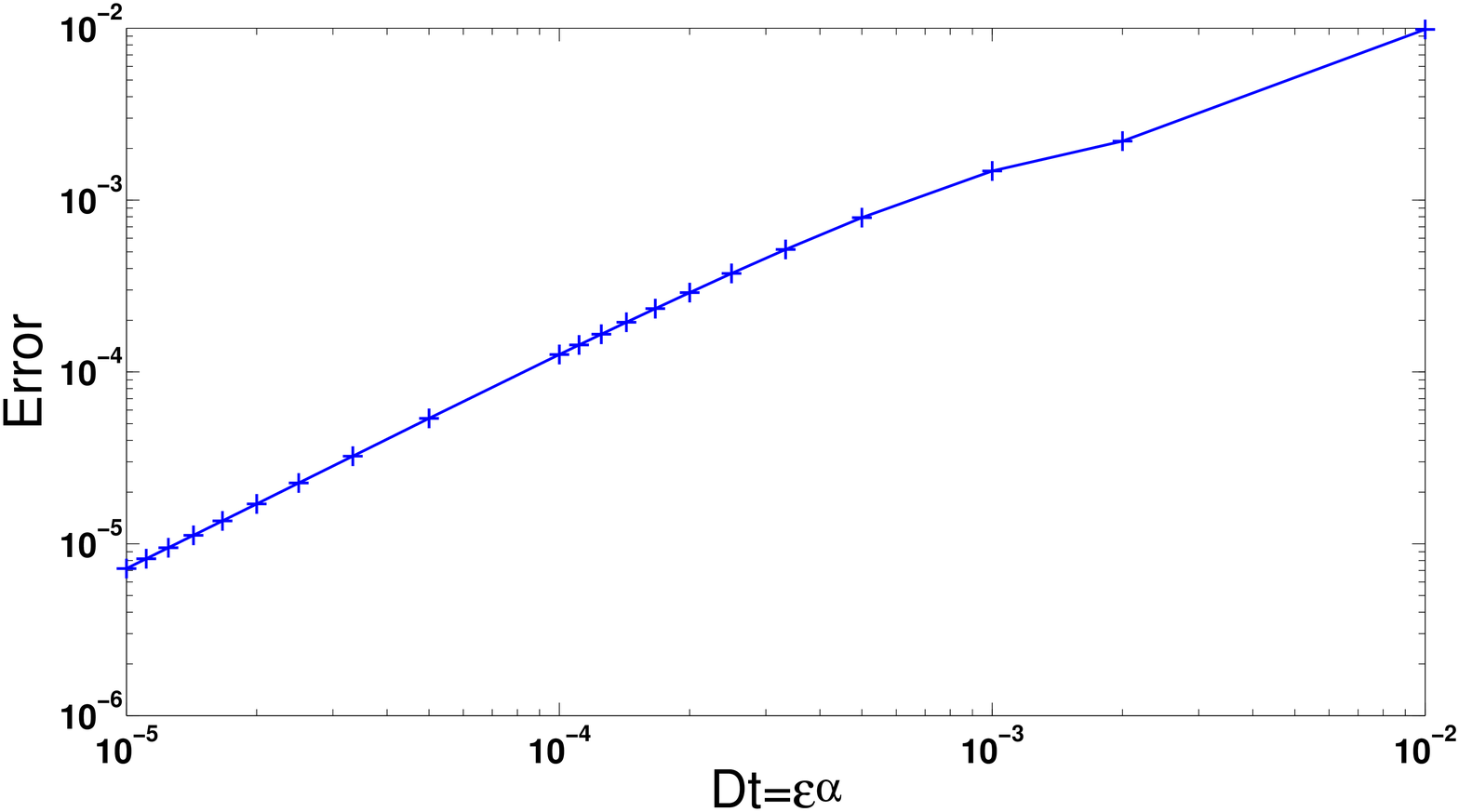}
\end{tabular}
\caption{Left: The difference between the DSA scheme computed for $\dt=10^{-5}$  and the DSA scheme computed for a range of $\dt$ as a function of $\ep$ (log scale). Right: The difference between the DSA scheme computed for $\dt=\ep^\alp$  and the DSA scheme computed for $\dt=10^{-7}$ as a function of $\dt$ (log scale).}
\label{a1moinsaAN_Unif}
\end{center}
\end{figure}

\section{Conclusion}

In this paper, we have first presented a formal derivation of diffusion and anomalous 
diffusion asymptotics from a BGK kinetic equation. This analysis enables us to understand the role of the large velocities induced by the heavy-tailed equilibrium in the anomalous diffusion asymptotics. 
Moreover, this formal derivation paves the way to construct three different numerical 
schemes for the kinetic equation, which all enjoy the asymptotic preserving property in both the diffusion and anomalous diffusion regimes. 

The asymptotic preserving character in the case of anomalous diffusion comes from the fact that 
we managed to take into account the effect of large velocities of the equilibrium. As we saw in the simplest 
case of a direct fully implicit scheme, the asymptotic is not preserved when these large velocities are truncated. 
The schemes using the micro-macro formulation as well as the Duhamel formulation of the equations 
require the use of the same trick on the velocities. Moreover, this last scheme seems to enjoy 
a uniform accuracy property which will be proved in  \cite{CrousHivertLemou2}.

In the near future, we aim at extending this work to more general context, considering 
higher dimensions, non periodic boundary conditions. The case of singular collision frequency also may generate anomalous diffusion (see \cite{Puel1}) and this also deserves a numerical study that we plan to do in a forthcoming work \cite{CrousHivertLemou}.

\newpage
\bibliographystyle{plain}
\bibliography{these}

\end{document}


\maketitle

\renewcommand{\thefootnote}{\fnsymbol{footnote}}
\footnotetext[5]{IRMAR. Universit\'e de Rennes $1$, Campus de Beaulieu. $35000$ Rennes }
\footnotetext[2]{INRIA-IPSO. Email : nicolas.crouseilles@inria.fr}
\footnotetext[3]{Email : helene.hivert@univ-rennes1.fr }
\footnotetext[4]{CNRS.  Email : mohammed.lemou@univ-rennes1.fr }

\renewcommand{\thefootnote}{\arabic{footnote}}

%

This file contains supplementary materials for the paper entitled \emph{Numerical \newline schemes for kinetic equations in the diffusion and anomalous diffusion limits. Part I: 
the case of heavy-tailed equilibrium}.

%
\section{General expression of the anomalous diffusion in the Fourier variable}
\label{Appendix1}

In this part, we will show the following proposition
\begin{proposition}
In case 2 (anomalous diffusion scaling), the equilibrium $M_\beta$ satisfies  
\[
\left\langle\left(\e^{-\ii\ep s k\cdot v} -1\right) M_\beta(v)\right\rangle\underset{{\ep \to 0 \atop s \to 0}\atop \text{independently}}
=(\ep |k|)^{\beta-d} \mathcal{C}(s)+(\ep s |k|)^{\beta-d}R(\ep,s,k),
\]
with $C(s)$ defined by
\begin{equation}
\label{Cs}
 \mathcal{C}(s)=\int_{\mathbb{R}^d} \left( \e^{-\ii s e\cdot w}-1\right)\frac{m}{|w|^\beta}\dd w,
 \end{equation}
 and where $R(\ep,s,k)$ is bounded and tends to $0$ for small 
$\ep$ and $s$ with an explicit form given by
\[
R(\ep,s,k)= \int_{w\in \mathbb{R}^d} \left( \e^{-\ii w\cdot e}-1 \right)\left( \frac{1}{(\ep s |k|)^\beta}M_\beta\left(\frac{w}{\ep s |k|}\right)-\frac{m}{|w|^\beta} \right)\dd w.
\]

\end{proposition}

\begin{proof}
Since we supposed that $M_\beta(v)\underset{\ep\to 0}{\sim}\frac{m}{|v|^\beta}$, 
there exists a continuous bounded function $\ell$ and a positive constant $K$ such that 
\begin{equation}
\label{ell}
\ell(v)\underset{|v|\to +\infty}\rightarrow 0 \text{~~and~~} M_\beta(v)-\frac{m}{|v|^\beta}=\ell(v)\frac{m}{|v|^\beta} \text{~if~} |v|>K.
\end{equation}
Hence we have
\[
\left\langle\left(\e^{-\ii\ep s k\cdot v} -1\right) M_\beta(v)\right\rangle =I_1+I_2+I_3,
\]
with  
\begin{align*}
I_1&=\int_{|v|\le K} \left(\e^{-\ii\ep s k\cdot v} -1\right) M_\beta(v) \dd v,   \\
I_2&=\int_{|v|> K} \left(\e^{-\ii\ep s k\cdot v} -1\right) \left(M_\beta(v)-\frac{m}{|v|^\beta}\right) \dd v, \\
I_3&=\int_{|v|> K} \left(\e^{-\ii\ep s k\cdot v} -1\right) \frac{m}{|v|^\beta} \dd v. 
\end{align*}
In  the sequel, we consider separately the equivalents of these three integrals to 
obtain the fractional diffusion term. 
Using the fact that $M_\beta$ is even, we can use a Taylor expansion for the exponential in the integral $I_1$
\begin{align*}
\left| I_1 \right| =\left|\int_{|v|\le K} \left(\e^{-\ii\ep s k\cdot v} -1\right) M_\beta(v) \dd v\right|&=\left|\int_{|v|\le K} \left(\e^{-\ii\ep s k\cdot v} -1+\ii \ep s k\cdot v\right) M_\beta(v) \dd v\right| \\
&\le \frac{\ep^2s^2}{2}\int_{|v|\le K} \left( k\cdot v\right)^2 M_\beta(v) \dd v \\
&\le \frac{|k|^2 |K|^2\ep^2s^2}{2}=o(\ep^\alp).
\end{align*}
For the integral $I_2$, we perform the change of variables $w=\ep |k| v$ for a nonzero $k$ and 
use the function $\ell$ in \eqref{ell}
\begin{align*}
I_2=\int_{|v|> K} \left(\e^{-\ii\ep s k\cdot v} -1\right)& \left(M_\beta(v)-\frac{m}{|v|^\beta}\right) \dd v =
\int_{|v|> K} \left(\e^{-\ii\ep s k\cdot v} -1\right) \ell(v)\frac{m}{|v|^\beta} \dd v \\
&=(\ep  |k|)^{\beta-d}\int_{|w|\ge \ep  K|k|} \left( \e^{-isw\cdot \frac{k}{|k|}}-1 \right) \ell\left( \frac{w}{\ep  |k|} \right) \frac{m}{|w|^\beta} \dd w \\
&=(\ep  |k|)^{\beta-d}\int_{w\in \mathbb{R}^d} \left( \e^{-isw\cdot \frac{k}{|k|}}-1 \right) \ell\left( \frac{w}{\ep  |k|} \right) \frac{m}{|w|^\beta} \dd w \\
&+(\ep  |k|)^{\beta-d}\int_{|w|< \ep  K|k|} \left( \e^{-isw\cdot \frac{k}{|k|}}-1 \right) \ell\left( \frac{w}{\ep  |k|} \right) \frac{m}{|w|^\beta} \dd w. 
\end{align*}
With the choice of $\beta=\alpha+d$, we know that those two integrals are well defined; 
indeed in a neighborhood of zero, the integrand is equivalent to $C|w|^{2-\beta}$ and 
at infinity we can dominate it by $\frac{C}{|w|^\beta}$. Since $\ell(v)\underset{|v|\to\infty}{\rightarrow} 0$,  
the dominated convergence theorem enables to write 
\[
I_2\underset{\ep\to 0}{=} o(\ep^{\beta-d}).
\]
We perform the same changes of variables in the integral $I_3$
\begin{align*}
I_3=\int_{|v|> K} \left(\e^{-\ii\ep s k\cdot v} -1\right) \frac{m}{|v|^\beta} \dd v&=(\ep  |k|)^{\beta-d}\int_{|w|\ge \ep  K|k|} \left( \e^{-isw\cdot \frac{k}{|k|}}-1 \right) \frac{m}{|w|^\beta} \dd w \\
&=(\ep  |k|)^{\beta-d}\int_{w\in \mathbb{R}^d} \left( \e^{-isw\cdot \frac{k}{|k|}}-1 \right)  \frac{m}{|w|^\beta} \dd w \\&
+(\ep  |k|)^{\beta-d}\int_{|w|< \ep  K|k|} \left( \e^{-isw\cdot \frac{k}{|k|}}-1 \right)  \frac{m}{|w|^\beta} \dd w. 
\end{align*}
With the same arguments as for $I_2$, the second integral tends to zero as $\ep$ goes to zero. 
The first one has rotational symmetry with respect to $k$ so, for any unitary vector $e$ of $\mathbb{R}^d$, we have 
\begin{align*}
I_3\underset{\ep\to 0}{\sim}&(\ep  |k|)^{\beta-d}\int_{w\in \mathbb{R}^d} \left( \e^{-isw\cdot e}-1 \right) \frac{m}{|w|^\beta} \dd w\\
 &=\ep^{\beta-d}|k|^{\beta-d}\mathcal{C}(s). 
\end{align*}
We finally get 
\[
\left\langle\left(\e^{-\ii\ep s k\cdot v} -1\right) M_\beta(v)\right\rangle\underset{\ep\to 0}{\sim}\ep^{\beta-d} 
|k|^{\beta-d}\int_{w\in \mathbb{R}^d} \left( \e^{-isw\cdot e}-1 \right) \frac{m}{|w|^\beta} \dd w=\ep^{\beta-d}
|k|^{\beta-d}\mathcal{C}(s),
\]
with $C(s)$ defined in (\ref{Cs}). We also need an estimation of the remainder in this equivalent 
when $\ep$ tends to zero and when $s$ is small (independently). We already have an upper bound 
for $I_1$; in order to get one for all the terms smaller than the equivalent we must specify the remainders 
in the integrals arising in the study of $I_2$ and $I_3$. Firstly, let us consider the second term at 
the end of the computation of $I_2$
\begin{align*}
(\ep  |k|)^{\beta-d}\int_{|w|< \ep  K|k|}& \left( \e^{-isw\cdot \frac{k}{|k|}}-1 \right) l\left( \frac{w}{\ep  |k|} \right) \frac{m}{|w|^\beta} \dd w
\\&\le c(\ep  |k|)^{\beta-d}\int_{|w|< \ep  K|k|} \left( \e^{-isw\cdot \frac{k}{|k|}}-1 \right) \frac{m}{|w|^\beta} \dd w \\
&\le c(\ep  |k|)^{\beta-d}s^2\int_{|w|< \ep  K|k|}|w|^{2-\beta}\dd w\\
&\le c(\ep  |k|)^{\beta-d}s^2 \int_0^{\ep K |k|} r^{d-1}r^{2-\beta}\dd r \\
&\le c s^2\ep^2 |k|^2,
\end{align*}
where $c$ is a constant depending on the upper bound of $\ell$, the constants $m$ 
and $\beta$, the dimension $d$ and $K$. We notice that the remainder in the integral $I_3$ 
can be bounded exactly in the same way, so it remains to consider the following term on which we 
perform the change of variable $w\rightarrow ws$
\begin{align*}
(\ep  |k|)^{\beta-d}\int_{w\in \mathbb{R}^d}& \left( \e^{-isw\cdot \frac{k}{|k|}}-1 \right) \ell\left( \frac{w}{\ep  |k|} \right) \frac{m}{|w|^\beta} \dd w
\\& = (\ep s |k|)^{\beta-d}\int_{w\in \mathbb{R}^d} \left( \e^{-\ii w\cdot \frac{k}{|k|}}-1 \right) \ell\left( \frac{w}{\ep s  |k|} \right) \frac{m}{|w|^\beta} \dd w. 
\end{align*} 
Using the same argument of dominated convergence, the integral has null limit for small $\ep$ and $s$, 
hence we get
\[
\left\langle\left(\e^{-\ii\ep s k\cdot v} -1\right) M_\beta(v)\right\rangle\underset{{\ep \to 0 \atop s \to 0}\atop \text{independently}}
=(\ep |k|)^{\beta-d} \mathcal{C}(s)+(\ep s |k|)^{\beta-d}R(\ep,s,k),
\]
where $R(\ep,s,k)$ is bounded and tends to $0$ for small $\ep$ and $s$ with an explicit form given by
\[
R(\ep,s,k)= \int_{w\in \mathbb{R}^d} \left( \e^{-\ii w\cdot e}-1 \right)\left( \frac{1}{(\ep s |k|)^\beta}M_\beta\left(\frac{w}{\ep s |k|}\right)-\frac{m}{|w|^\beta} \right)\dd w.
\]
\qquad \end{proof}

\section{General expression of the anomalous diffusion in the original space variable}
\label{Appendix2}
In this part, we will show the following proposition which enables to derive 
the anomalous diffusion limit of the kinetic equation with the original space variable 
in the general case.
\begin{proposition}
Considering an equilibrium $M_\beta(v)$ in the case 2 (anomalous diffusion scaling), 
the following function $a(\ep, z)$ 
\[
a(\ep,z)=\int_0^{\frac{t}{\ep^\alp}}|z|^\beta \frac{\e^{-s}}{(\ep s)^d}M_\beta\left(\frac{z}{\ep s}\right)\dd s.
\]
satisfies 
\[
a(\ep,z)\underset{\ep\to 0}{\sim} m\ep^{\beta-d}\int_0^{+\infty}s^{\beta-d}\e^{-s}\dd s,
\]
and 
\[
a(\ep,z)-m\ep^{\beta-d}\int_0^{+\infty}s^{\beta-d}\e^{-s}\dd s\underset{\ep\to 0}{\le} C \ep^{\beta-d},
\]
where $C$ is a constant independent of $\ep$. 
\end{proposition}

\begin{proof}
We start with the equality 
\[
a(\ep,z)=\int_0^{+\infty}|z|^\beta \frac{\e^{-s}}{(\ep s)^d}M_\beta\left(\frac{z}{\ep s}\right)\dd s-\int_{\frac{t}{\ep^\alp}}^{+\infty}|z|^\beta \frac{\e^{-s}}{(\ep s)^d}M_\beta\left(\frac{z}{\ep s}\right)\dd s,
\]
and as $M_\beta$ is bounded by a constant $C$, we have
\begin{align*}
\left| \int_{\frac{t}{\ep^\alp}}^{+\infty}|z|^\beta\frac{\e^{-s}}{(\ep s)^d}M_\beta\left(\frac{z}{\ep s}\right)\dd s\right|& \le C \left|z\right|^\beta \frac{1}{\ep^d}\int_{\frac{t}{\ep^\alp}}^{+\infty}\frac{\e^{-s}}{s^d}\dd s \underset{\ep\to 0}= O(\ep^\infty). 
\end{align*}
Hence, we obtain 
\[
\frac{1}{\ep^{\beta-d}}\left| \int_{\frac{t}{\ep^\alp}}^{+\infty}|z|^\beta\frac{\e^{-s}}{(\ep s)^d}M_\beta\left(\frac{z}{\ep s}\right)\dd s\right|\underset{\ep\to 0}\longrightarrow 0, 
\]
and there exists a constant $C_1$ such that 
\[
\left| \int_{\frac{t}{\ep^\alp}}^{+\infty}|z|^\beta \frac{\e^{-s}}{(\ep s)^d}M_\beta\left(\frac{z}{\ep s}\right)\dd s\right| \le \ep^{\beta-d}C_1.
\]
As we did in the previous appendix, we consider a function $\ell(v)$ satisfying \eqref{ell}. 
Considering that, for a nonzero $z$, 
\[
\left| \frac{z}{\ep s} \right|>K \Leftrightarrow s<\frac{\left| z\right|}{\ep K},
\]
$a(\ep,z)$ becomes
\begin{align}
a(\ep,z)&=\int_0^{+\infty}\left|z\right|^\beta \frac{\e^{-s}}{\left(\ep s\right)^d}\frac{m}{\left|\frac{z}{\ep s}\right|^\beta}\dd s 
+\int_0^{\frac{\left|z\right|}{\ep K}}\left|z\right|^\beta \frac{\e^{-s}}{\left(\ep s\right)^d}\left(M_\beta\left( \frac{z}{\ep s} \right) - \frac{m}{\left|\frac{z}{\ep s}\right|^\beta}\right)\dd s \nonumber\\
&+\int_{\frac{\left|z\right|}{\ep K}}^{+\infty}\left|z\right|^\beta \frac{\e^{-s}}{\left(\ep s\right)^d}M_\beta\left( \frac{z}{\ep s} \right) \dd s 
-\int_{\frac{|z|}{\ep K}}^{+\infty}\left|z\right|^\beta \frac{\e^{-s}}{\left(\ep s\right)^d}\frac{m}{\left|\frac{z}{\ep s}\right|^\beta}\dd s \nonumber\\
\label{aepsz}
&-\int_{\frac{t}{\ep^\alp}}^{+\infty}|z|^\beta \frac{\e^{-s}}{(\ep s)^d}M_\beta\left(\frac{z}{\ep s}\right)\dd s.
\end{align}
Now we consider separately these five integrals of \eqref{aepsz}. 
The first one is exactly the equivalent we are looking for
\[
\int_0^{+\infty}\left|z\right|^\beta \frac{\e^{-s}}{\left(\ep s\right)^d}\frac{m}{\left|\frac{z}{\ep s}\right|^\beta}\dd s=m\ep^{\beta-d}\int_0^{+\infty}s^{\beta-d}\e^{-s}\dd s,
\]
so we must prove that the other terms of \eqref{aepsz} are small. For the second integral of \eqref{aepsz}, 
we use the function $\ell$ introduced in \eqref{ell} 
\begin{align*}
\int_0^{\frac{\left|z\right|}{\ep K}}\left|z\right|^\beta \frac{\e^{-s}}{\left(\ep s\right)^d}\left(M_\beta\left( \frac{z}{\ep s} \right) - \frac{m}{\left|\frac{z}{\ep s}\right|^\beta}\right)\dd s 
&=\int_0^{\frac{\left|z\right|}{\ep K}}\left|z\right|^\beta \frac{\e^{-s}}{\left(\ep s\right)^d}\ell\left( \frac{z}{\ep s} \right) \frac{m}{\left|\frac{z}{\ep s}\right|^\beta}\dd s \\
&=\ep^{\beta-d}\int_0^{+\infty}s^{\beta-d}\e^{-s}\ell\left(\frac{z}{\ep s}  \right)\mathbf{1}_{\left[0,\frac{\left|z\right|}{\ep K}\right]}(s)\dd s. 
\end{align*}
As $\ell(v)\underset{|v|\to+\infty}\longrightarrow0$, we have
\[
s^{\beta-d}\e^{-s}\ell\left(\frac{z}{\ep s}  \right)\mathbf{1}_{\left[0,\frac{\left|z\right|}{\ep K}\right]}(s)\dd s\underset{\ep\to 0}\longrightarrow0,
\]
and
\[
\left| s^{\beta-d}\e^{-s}\ell\left(\frac{z}{\ep s}  \right)\mathbf{1}_{\left[0,\frac{\left|z\right|}{\ep K}\right]}(s)\dd s\right|\le  C s^{\beta-d}\e^{-s}, 
\]
where the constant $C$ denotes the maximum of $\ell$. 
So, by the dominated convergence theorem 
\[
\frac{1}{\ep^{\beta-d}}\int_0^{\frac{\left|z\right|}{\ep K}}\left|z\right|^\beta \frac{\e^{-s}}{\left(\ep s\right)^d}\left(M_\beta\left( \frac{z}{\ep s} \right) - \frac{m}{\left|\frac{z}{\ep s}\right|^\beta}\right)\dd s \underset{\ep\to 0}\longrightarrow0,
\]
and there exists a constant $C_2$ such that 
\[
\left| \int_0^{\frac{\left|z\right|}{\ep K}}\left|z\right|^\beta \frac{\e^{-s}}{\left(\ep s\right)^d}\left(M_\beta\left( \frac{z}{\ep s} \right) - \frac{m}{\left|\frac{z}{\ep s}\right|^\beta}\right)\dd s \right| \le C_2 \ep^{\beta-d}.
\]
Then, we remark that the third and fourth integral in \eqref{aepsz} are exponentially small 
\begin{align*}
&\left| \int_{\frac{\left|z\right|}{\ep K}}^{+\infty}\left|z\right|^\beta \frac{\e^{-s}}{\left(\ep s\right)^d}M_\beta\left( \frac{z}{\ep s} \right) \dd s 
-\int_{\frac{|z|}{\ep K}}^{+\infty}\left|z\right|^\beta \frac{\e^{-s}}{\left(\ep s\right)^d}\frac{m}{\left|\frac{z}{\ep s}\right|^\beta}\dd s \right| 
\\&\le C \frac{|z|^\beta}{\ep^d}\int_{\frac{\left|z\right|}{\ep K}}^{+\infty}\frac{\e^{-s}}{s^d}\dd s +\ep^{\beta-d}m\int_{\frac{\left|z\right|}{\ep K}}^{+\infty}\e^{-s}s^{\beta-d}\dd s \\
&= O(\ep^{\infty}),
\end{align*}
where $C$ denotes the maximum of $M_\beta$. Hence, we deduce 
\[
\frac{1}{\ep^{\beta-d}} \left(\int_{\frac{\left|z\right|}{\ep K}}^{+\infty}\left|z\right|^\beta \frac{\e^{-s}}{\left(\ep s\right)^d}M_\beta\left( \frac{z}{\ep s} \right) \dd s 
-\int_{\frac{|z|}{\ep K}}^{+\infty}\left|z\right|^\beta \frac{\e^{-s}}{\left(\ep s\right)^d}\frac{m}{\left|\frac{z}{\ep s}\right|^\beta}\dd s\right)\underset{\ep\to 0}\longrightarrow 0
\]
and there exists a constant $C_3$ such that
\[
\left| \int_{\frac{\left|z\right|}{\ep K}}^{+\infty}\left|z\right|^\beta \frac{\e^{-s}}{\left(\ep s\right)^d}M_\beta\left( \frac{z}{\ep s} \right) \dd s 
-\int_{\frac{|z|}{\ep K}}^{+\infty}\left|z\right|^\beta \frac{\e^{-s}}{\left(\ep s\right)^d}\frac{m}{\left|\frac{z}{\ep s}\right|^\beta}\dd s \right| \le \ep^{\beta-d}C_3.
\]
Since the last integral of \eqref{aepsz} has already been treated, this conclude the proof. 

\qquad \end{proof}